\documentclass[12pt]{amsart}

\usepackage{amssymb,color}

\newtheorem{theorem}[subsection]{Theorem}
\newtheorem{lemma}[subsection]{Lemma}
\newtheorem{remark}[subsection]{Remark}

\newcommand{\field}{{\bf F}}
\newcommand{\rfield}{{\bf k}}
\newcommand{\rightquot}{/\!\!/}
\newcommand{\zp}{{{\bf{Z}}/p}}
\newcommand{\tat}{t\^ete-\`a-t\^ete}
\newcommand{\bg}{\bar{\gamma}}

\DeclareMathOperator{\LM}{\mathrm{LM}}
\DeclareMathOperator{\LT}{\mathrm{LT}}

\title[Elementary abelian $p$-groups]
{Rings of invariants for the three  dimensional modular representations
of elementary abelian $p$-groups of rank four}

\author{Th\'eo Pierron  }
\address{D\'epartement  Math\'ematiques\\
\hfil\break\indent ENS Rennes, 35170 BRUZ, France }
\email{Theo.Pierron@ens-rennes.fr}

\author{R.J. Shank}
\address{School of Mathematics, Statistics \&  Actuarial Science \\
\hfil\break\indent University of Kent, Canterbury, CT2 7NF, UK}
\email{R.J.Shank@kent.ac.uk}

\subjclass{13A50}
\date{\today}

\begin{document}
\begin{abstract}
We show that the rings of invariants for the three
dimensional modular representations of an elementary abelian $p$-group of rank four are
complete intersections with embedding dimension at most five. Our results confirm
the conjectures of Campbell, Shank and Wehlau~\cite[\S 8]{CSW} 
for these representations.
\end{abstract}

\maketitle

\section*{Introduction}

We continue the investigation of the rings of invariants of modular
representation of elementary abelian $p$-groups initiated
in \cite{CSW}. We show that the rings of invariants for three
dimensional modular representations of groups of rank four are
complete intersections and we confirm the conjectures
of~\cite[\S 8]{CSW} for these representations.

Let $V$ denote an $n$ dimensional representation of a group $G$, over a field $\field$ of characteristic $p$,
for a prime number $p$. We will usually assume that $G$ is finite and that $p$ divides the order of $G$,
in other words, that $V$ is a {\it modular representation} of $G$.
We view $V$ as a left module over the group ring $\field G$ and the dual, $V^*$,
as a right $\field G$-module. Let $\field[V]$ denote the symmetric algebra on $V^*$.
The action of $G$ on $V^*$ extends to an action by 
degree preserving algebra automorphisms  on $\field[V]$.
By choosing a basis $\{x_1,x_2,\ldots,x_n\}$ for $V^*$, we identify $\field[V]$ with the algebra
of polynomials $\field[x_1,x_2,\ldots,x_n]$.
Our convention that $\field[V]$ is a right $\field G$-module is consistent
with the convention used by the invariant theory package in the computer algebra 
software Magma~\cite{magma}. The ring of invariants, $\field[V]^G$, is the subring
of $\field[V]$ consisting of those polynomials fixed by the action of $G$.
Note that elements of $\field[V]$ represent polynomial functions on $V$ and
that elements of $\field[V]^G$ represent polynomial functions on the set of orbits
$V/G$. For $G$ finite and $\field$ algebraically closed, $\field[V]^G$ is the
ring of regular functions on the categorical quotient $V\rightquot G$.
For background on the invariant theory of finite groups, see \cite{Benson},
\cite{CW}, \cite{DK}, or \cite{NS}.

Computing the ring of invariants for a modular representation is typically a difficult problem;
the rings are often not Cohen-Macaulay. It is natural to take $p$-groups as a starting point and
recent work of David Wehlau \cite{wehlau} gives us a good understanding in the case of a cyclic 
group of order $p$. The next step is to look at elementary abelian $p$-groups.
The rings of invariants for the two dimensional modular representations  of elementary abelian
$p$-groups were computed in Section~2 of \cite{CSW}
and the three dimensional modular representations were classified in Section~4 of that paper. %~\cite[\S 4]{CSW}. 
The only three dimensional
representations for which computing the ring of invariants is not
straight-forward are those of type $(1,1,1)$, in other words, those
representations for which $\dim(V^G)=1$ and $\dim((V/V^G)^G)=1$. Our goal here
is to compute the rings of invariants for representations of
type $(1,1,1)$ for groups of rank four.  
The methods we use are essentially the same as the methods used in \cite{CSW}. 
As the rank increases the complexity of the required calculations increases;
we believe that it is not feasible to use the methods here for rank greater than four. 

We denote by
$E=\langle e_1,e_2,e_3,e_4\rangle\cong (\zp)^4$ a rank four elementary abelian $p$-group.
Note that $E$ only has representations of 
type $(1,1,1)$ if $p>2$, so we make this assumption throughout the paper.
As in Section~4 of \cite{CSW},
define $\sigma:\field^2\to {\rm GL}_3(\field)$ by
$$\sigma(c_1,c_2):=\begin{pmatrix}%\left[ \begin{array}{ccc} 
1&2c_1&c_1^2+c_2 \\0&1&c_1 \\0&0&1 
\end{pmatrix}%\end{array}\right]
.$$
Note that $\sigma$ defines a representation of the group $(\field^2,+)$. 
For a matrix
$$M:=\begin{pmatrix} %\left[\begin{array}{cccc}
c_{11}&c_{12}&c_{13}&c_{14}\\
c_{21}&c_{22}&c_{23}&c_{24}
\end{pmatrix}%\end{array}\right]
$$
with $c_{ij}\in\field$, the assignment $e_j\mapsto
\sigma(c_{1j},c_{2j})$ determines a three dimensional representation
of $E$, which we denote by $V_M$. The action of $E$ on $\field[x,y,z]$ is given by right
multiplication on $x=[0\, 0\, 1]$, $y=[0\, 1\, 0]$ and
$z=[1\,0\,0]$. Thus $x\cdot\sigma(c_1,c_2)=x$, $y\cdot\sigma(c_1,c_2)=y+c_1x$ and 
$z\cdot\sigma(c_1,c_2)=z+2c_1y+(c_1^2+c_2)x$.
The representation $V_M$ is of type $(1,1,1)$ if at least one
$c_{1j}$ is non-zero. Furthermore, by Proposition~4.1 of \cite{CSW},
for every representation of type $(1,1,1)$, there exists a choice of
basis for which the action is given by some matrix $M$.

In this paper, we compute $\field[V_M]^E$ for all $M\in\field^{2\times 4}$.
We give a stratification of $\field^{2\times 4}$ and show that within each stratum
there is a uniform computation of $\field[V_M]^E$. Note that the automorphism group of $E$
is isomorphic to ${\rm GL}_4(\field_p)$, where $\field_p$ denotes the field of $p$ elements.
Since $\field_p\subseteq \field$, there is a natural right action of ${\rm GL}_4(\field_p)$
on $\field^{2\times 4}$. If $M$ and $M'$ lie in the same ${\rm GL}_4(\field_p)$-orbit,
then $\field[V_M]^E=\field[V_{M'}]^E$. Essentially, we study subrings of $\field[x,y,z]$ parametrised by points in
$\field^{2\times 4}/{\rm GL}_4(\field_p)$ and use elements of $\field[\field^{2\times 4}]^{{\rm SL}_4(\field_p)}$ to describe the
stratification.

In Section~\ref{sec:genericcase}, we work over the field $\rfield:=\field_p(x_{ij}\mid i\in\{1,2\},j\in\{1,2,3,4\})$ and compute
$\rfield[V_{\mathcal M}]^E$ for the generic matrix
$$\mathcal M:=\begin{pmatrix}%\left[\begin{array}{cccc}
x_{11}&x_{12}&x_{13}&x_{14}\\
x_{21}&x_{22}&x_{23}&x_{24}
\end{pmatrix}%\end{array}\right]
.$$
We show that $\rfield[V_{\mathcal M}]^E$ is a complete intersection of embedding dimension five with
generators in degrees $1$, $p^2$, $p^2+2p$, $p^3+2$ and $p^4$, and relations in degrees
$p^3+2p^2$ and $p^4+2p$. 
Consider the $10\times 4$ matrix
\[\Gamma:=\begin{pmatrix} 
           x_{11}&&x_{12} &&x_{13} && x_{14}\\
           x_{21}&&x_{22} &&x_{23} && x_{24}\\
           x_{11}^p&&x_{12}^p &&x_{13}^p && x_{14}^p\\
           x_{21}^p&&x_{22}^p &&x_{23}^p && x_{24}^p\\
           &&&\vdots &&&\\
           x_{11}^{p^4}&&x_{12}^{p^4} &&x_{13}^{p^4} && x_{14}^{p^4}\\
           x_{21}^{p^4}&&x_{22}^{p^4} &&x_{23}^{p^4} && x_{24}^{p^4}\\
\end{pmatrix}\]
and for a subsequence $(i,j,k,\ell)$ of $(1,2,\ldots,10)$, let $\gamma_{ijk\ell}$ denote the associated 
$4\times 4$ minor of $\Gamma$. Note that $\gamma_{ijk\ell}\in\field[\field^{2\times 4}]^{{\rm SL}_4(\field_p)}$ 
and,  for $g\in{\rm GL}_4(\field_p)$, we have $g(\gamma_{ijk\ell})=\det(g)\gamma_{ijk\ell}$. 
We use zero-sets of various $\gamma_{ijk\ell}$ to define the stratification
of $\field^{2\times 4}/{\rm GL}_4(\field_p)$.
In Section~\ref{sec:1357!=0,1234=!0,1235!=0}, we show that for $M\in\field^{2\times 4}$ with
$\gamma_{1234}(M)\not=0$, $\gamma_{1235}(M)\not=0$, and $\gamma_{1357}(M)\not=0$, the generic calculation
survives evaluation. In Sections~$4$ through $10$, we compute the rings of 
invariants for the remaining strata.
\begin{itemize}

\item[\S 4.] For $\gamma_{1357}(M)\not=0$, $\gamma_{1235}(M)\not=0$, $\gamma_{1234}(M)=0$: $\field[V_M]^E$ 
is a complete intersection with generators in degrees $1$, $2p$, $p^3$, $p^3+2$ and $p^4$,
and relations in degrees $2p^3$ and $p^4+2p$.

\item[\S 5.] For $\gamma_{1357}(M)\not=0$, $\gamma_{1235}(M)=0$, $\gamma_{1234}(M)\not=0$:
If $\gamma_{1245}(M)\not=0$ then
$\field[V_M]^E$ is a complete intersection with generators in 
degrees $1$, $p^2$, $p^2+p$, $p^3+p+2$ and $p^4$,
and relations in degrees $p^3+p^2$ and $p^4+p^2+2p$. 
Otherwise, $\field[V_M]^E$ is a hypersurface with generators in degrees
$1$, $p^2$, $p^2+2$ and $p^4$, with the relation in degree $p^4+2p^2$.

\item[\S 6.] For $\gamma_{1357}(M)=0$, $\gamma_{1235}(M)\not=0$, $\gamma_{1234}(M)\not=0$: 
$\field[V_M]^E$ is a complete intersection with generators in 
degrees $1$, $p^2$, $p^2+2p$, $p^3+1$ and $p^4$,
and relations in degrees $p^3+2p^2$ and $p^4+p$.

\item[\S 7.] For $\gamma_{1357}(M)\not=0$, $\gamma_{1235}(M)=0$, $\gamma_{1234}(M)=0$: 
$\field[V_M]^E$ is a hypersurface.
If $\gamma_{1257}(M)=0$, then the generators are in degrees 
$1$, $2$, $p^4$ and $p^4$ and the relation is in degree $2p^4$. 
Otherwise, the generators are in degrees $1$, $p$, $p^3+p^2+p+2$, $p^4$ and the relation 
is in degree $p^4+p^3+p^2+2p$.

\item[\S 8.] For $\gamma_{1357}(M)=0$, $\gamma_{1235}(M)\not=0$, $\gamma_{1234}(M)=0$: 
$\field[V_M]^E$ is a complete intersection with generators in degrees
$1$, $2p$, $p^3$, $p^3+1$, and $p^4$, with relations in degrees $2p^2$ and $p^4+p$.

\item[\S 9.] For $\gamma_{1357}(M)=0$, $\gamma_{1235}(M)=0$, $\gamma_{1234}(M)\not=0$: 
If $\gamma_{1245}(M)\not=0$, then $\field[V_M]^E$ is a complete intersection with generators in degrees
$1$, $p^2$, $p^2+p$, $p^3+1$, and $p^4$, with relations in degrees $p^3+p^2$ and $p^4+p$.
Otherwise, $\field[V_M]^E$ is a hypersurface with generators in degrees $1$, $p^2$, $p^2+1$
and $p^4$, with a relation in degree $p^4+p^2$. 

\item[\S 10.] For $\gamma_{1357}(M)=0$, $\gamma_{1235}(M)=0$, $\gamma_{1234}(M)=0$: 
If $\gamma_{1246}(M)\not=0$ then $\field[V_M]^E$ is a hypersurface with generators in
degrees $1$, $p$, $p^3+1$, $p^4$, and a relation in degree $p^4+p$. 
Otherwise, the representation is either not faithful or not of type $(1,1,1)$; 
in either case, the invariants were computed in \cite{CSW}.
\end{itemize}

\section{Preliminaries}\label{prelim}

We make extensive use of the theory of SAGBI bases to compute rings of
invariants.  A SAGBI basis is the {\bf S}ubalgebra {\bf A}nalogue of a
{\bf G}r\"obner {\bf B}asis for {\bf I}deals, and is a particularly nice 
generating set for the subalgebra. The concept was
introduced independently by Robbiano-Sweedler~\cite{rs} and
Kapur-Madlener~\cite{km}; a useful reference is Chapter~11 of
Sturmfels~\cite{sturmfels}.  We adopt the convention that a monomial
is a product of variables and a term is a monomial with a coefficient.
We use the graded reverse lexicographic order with $x<y<z$.  For a
polynomial $f\in\field[x,y,z]$, we denote the lead monomial of $f$ by
$\LM(f)$ and the lead term of $f$ by $\LT(f)$.  For
$\mathcal{B}=\{h_1,\ldots,h_{\ell}\}\subset \field[x,y,z]$ and
$I=(i_1,\ldots,i_{\ell})$, a sequence of non-negative integers, denote
$\prod_{j=1}^{\ell}h_j^{i_j}$ by $h^I$. A {\it \tat\ } for $\mathcal{B}$
is a pair $(h^I,h^J)$ with $\LM(h^I)=\LM(h^J)$; we say that a \tat\ is
{\it non-trivial} if the support of $I$ is disjoint from the support
of $J$.  The reduction of an $S$-polynomial is a fundamental
calculation in the theory of Gr\"obner bases.  The analogous
calculation for SAGBI bases is the {\it subduction} of a \tat.
For any $f\in\field[x,y,z]$, if there exists a sequence $I$ such that
$\LM(f)=\LM(h^I)$, we can choose $c\in\field$ so that $\LT(f)=\LT(ch^I)$.
Then $\LT(f-ch^I)<\LT(f)$. If by iterating this process we can write $f$
as a polynomial in the $h_i$, we say that $f$ subducts to zero (using $\mathcal B$).
For a \tat\ $(h^I,h^J)$, choose $c$ so that $\LT(h^I)=\LT(ch^J)$. We say that the \tat\ subducts to zero if 
$h^I-ch^J$ subducts to zero.
A subset $\mathcal{B}$ of a subalgebra $A\subset\field[x_1,\ldots,x_n]$
is a SAGBI basis for $A$ if the lead monomials of the elements of
$\mathcal{B}$ generate the lead term algebra of $A$ or, equivalently,
every non-trivial \tat\ for $\mathcal{B}$ subducts to zero. For
background material on term orders and Gr\"obner bases, we recommend
\cite{AL}.

The following specialisation of Theorem~1.1 of \cite{CSW} is our primary computational tool.
Note that under the hypotheses of the theorem, $\{x,h_1,h_{\ell}\}$ is a homogeneous system of 
parameters and, therefore, $\field[V_M]^E$ is an integral extension of $A$. 

\begin{theorem} \label{thm:compute}
For homogeneous $h_1,\ldots,h_{\ell}\in\field[V_M]^E$ with $\LM(h_1)=y^i$ for some $i>0$, $\LM(h_{\ell})=z^j$
for some $j>0$ and $\LM(h_k)\in\field[y,z]$ for $k=2,\ldots,\ell-1$, 
define $\mathcal B:=\{x,h_1,\ldots,h_{\ell}\}$ and let $A$ denote the algebra generated by 
$\mathcal B$. If $A[x^{-1}]=\field[V_M]^E[x^{-1}]$ and $\mathcal B$ is a SAGBI basis for $A$, then
$A=\field[V_M]^E$ and $\mathcal B$ is a SAGBI basis for $\field[V_M]^E$.
\end{theorem}

Note that, if an algebra is generated by a finite SAGBI basis, then for the corresponding presentation, 
the ideal of relations is generated by elements  corresponding to the subductions of the 
non-trivial \tat s (see Corollary~11.6 of \cite{sturmfels}). 
We use the term {\it complete intersection} to refer to an algebra 
with a presentation for which the ideal of relations is generated by a regular sequence. 
Since the Krull dimension of  $\field[V_M]^E$ is three, the ring is a complete intersection 
if the number of generators minus the number of non-trivial \tat s is three.   

We routinely use the {\it SAGBI/Divide-by x} algorithm introduced in Section~1 of \cite{CSW}.
The traditional SAGBI basis algorithm proceeds by subducting \tat s and adding any non-zero subductions to the generating set.
For SAGBI/Divide-by-$x$, if a non-zero subduction is divisible by $x$, we divide by the highest possible power of $x$
before adding the polynomial to the generating set.
While the SAGBI algorithm extends the generating set for a given subalgebra, SAGBI/Divide-by-$x$ extends the subalgebra.
If we start with a subalgebra $A$ which contains a homogeneous system of parameters and 
satisfies the condition that $A[x^{-1}]=\field[V_M]^E[x^{-1}]$, then the SAGBI/Divide-by-$x$ algorithm will produce a generating set
for $\field[V_M]^E$ (see Theorem~1.2 of \cite{CSW}).

For $f\in \field[V_M]$, we define the {\it norm} of $f$ to be the orbit product
$$N_M(f):=\prod\{f\cdot g \mid g\in E\}\in\field[V_M]^E$$
with the action of $E$ determined by $M$. When applying Theorem~\ref{thm:compute},
we often take $h_{\ell}$ to be $N_M(z)$.
\begin{remark}\label{rem:norm}
 Note that the action of $E$ restricts to an action on
$\field[x,y]$ and that $\field[x,y]^E=\field[x,N_M(y)]$ (see Section~2 of \cite{CSW}).
Therefore, if $h\in\field[x,y]^E$ is homogeneous with $\deg(h)=|\{y\cdot g\mid g\in E\}|$ then
$h$ is a linear combination of $N_M(y)$  and $x^{\deg(h)}$.
\end{remark}

Define $\delta:=y^2-xz$ and observe that 
$$\delta\cdot\sigma(c_1,c_2)=(y+c_1x)^2-x(z+2c_1y+(c_1^2+c_2)x)=\delta-c_2x^2.$$ 
Note that $\field[x,y,z][x^{-1}]=\field[x,y,-\delta/x][x^{-1}]$ and that the
$\field[x,y,-\delta/x]^E$ is a polynomial algebra (see Theorem~3.9.2 of \cite{CW}).
This ``change of basis'' can be a useful way to compute the field of fractions of $\field[V_M]^E$.
Form the matrix $\widetilde{\Gamma}$ by augmenting $\Gamma$ with the column
$$\left[\frac{y}{x}\;\; \left(-\frac{\delta}{x^2}\right)\;\; \left(\frac{y}{x}\right)^p\;\; 
\left(-\frac{\delta}{x^2}\right)^p\; \cdots\: \left(\frac{y}{x}\right)^{p^4}\;  \left(-\frac{\delta}{x^2}\right)^{p^4}\right]^T.$$
For a subsequence $J=(j_1,\ldots,j_5)$ of $(1,2,\ldots,10)$, let $\widetilde{f}_J\in\rfield[x,y,z][x^{-1}]$
denote the associated $5\times 5$ minor of $\widetilde{\Gamma}$. Let $f_J$ denote the element of
$\rfield[x,y,z]$ constructed by minimally clearing the denominator of $\widetilde{f}_J$. Observe that 
$f_J\in\rfield[V_{\mathcal M}]^E$. Furthermore, the coefficients of $f_J$ lie in 
$\field_p[x_{ij}]^{\rm SL_4(\field_p)}$ and, for an arbitrary $M\in\field^{2\times 4}$, evaluating the coefficients of
$f_J$ at $M$ gives an element $\overline{f_J}\in\field[V_M]^E$. Invariants constructed in this way are a crucial ingredient in our calculations.
 Define $f_1:=f_{12345}$ and observe that 
$\LT(f_1)=\gamma_{1234}y^{p^2}$. Note that $\LT(f_{12346})=-\gamma_{1234}y^{2p^2}$. A straight-forward calculation shows that
$$\LT(f_1^2+\gamma_{1234}f_{12346})=2\gamma_{1234}\gamma_{1235}x^{p^2-2p}y^{p^2+2p}.$$ Therefore, 
$$f_2:=\frac{f_1^2+\gamma_{1234}f_{12346}}{2x^{p^2-2p}}\in\rfield[V_{\mathcal M}]^E$$
has lead term $\gamma_{1234}\gamma_{1235}y^{p^2+2p}$.

We make frequent use of the {\it Pl\"ucker relations} for the minors of $\Gamma$ and $\widetilde{\Gamma}$.

\begin{theorem}\label{thm:plucker}
For $N$ an $n\times m$ matrix with $n>m$, let $p_{i_1,\ldots,i_m}$ denote the $m\times m$ 
minor of $N$ determined by the rows $i_1,\ldots,i_m$. For sequences $(i_1,\ldots,i_{m-1})$
and $(j_1,\ldots,j_{m+1})$, we have the following Pl\"ucker relation
$$
\sum_{a=1}^{m+1}(-1)^a p_{i_1,\ldots,i_{m-1},j_a}p_{j_1,\ldots,j_{a-1},j_{a+1},\ldots,j_{m+1}}=0.
$$
\end{theorem}

\noindent For a proof of the above theorem, see, for example, \cite[\S 4.1.3]{LR}.

\begin{lemma}\label{lem:first-plucker}
For $2<i<7$,
$$\gamma_{12i7}\gamma_{1234}^p=\gamma_{12i6}\gamma_{1235}^p-\gamma_{12i5}\gamma_{1245}^p+\gamma_{12i4}\gamma_{1345}^p-\gamma_{12i3}\gamma_{2345}^p.$$
\end{lemma}
\begin{proof}
Since taking $p^{th}$ powers is $\field_p$-linear, $\gamma_{(i+2)(j+2)(k+2)(\ell+2)}=\gamma_{ijk\ell}^p$. For example, $\gamma_{3456}=\gamma_{1234}^p$.
The result follows from this fact, using the $(1,2,i)(3,4,5,6,7)$ Pl\"ucker relation for the matrix $\Gamma$. 
\end{proof}

For $K=(k_1,k_2,\ldots,k_6)$ a subsequence of $(1,2,\ldots,10)$, let $K_i$ denote the subsequence of $K$ 
formed by omitting $i$ and let $K_{i,j}$ denote the subsequence of $K$ formed by omitting $i$ and $j$.
The following is Lemma 5.3 from \cite{CSW}.
\begin{lemma}\label{lem:field}
For any subsequence $(i_1,i_2,i_3)$ of $K$,
$$(-1)^{\epsilon_1}\gamma_{K_{i_1,i_2}}\widetilde{f}_{K_{i_3}}
+(-1)^{\epsilon_2}\gamma_{K_{i_2,i_3}}\widetilde{f}_{K_{i_1}}
+(-1)^{\epsilon_3}\gamma_{K_{i_1,i_3}}\widetilde{f}_{K_{i_2}}=0$$
for some choice of $\epsilon_{\ell}\in\{0,1\}$.
\end{lemma}

\begin{remark} Note that $\gamma_{1357}(M)=0$ if and only if $\{c_{11},c_{12},c_{13},c_{14}\}$ is linearly dependent over $\field_p$.
This follows from the usual construction of the Dickson invariants, see for example \cite{wilk}. The key observation is that
 $\gamma_{1357}(M)^{p-1}$ is the product of the non-zero $\field_p$-linear combinations of $\{c_{11},c_{12},c_{13},c_{14}\}$.
\end{remark}

\section{The Generic Case}
\label{sec:genericcase}
 In this section we compute $\rfield[V_{\mathcal M}]^E$. 
With $f_1$ and $f_2$ defined as in Section~\ref{prelim}, 
using Theorem~5.2 of \cite{CSW},
we see that $\rfield[V_{\mathcal M}]^E[x^{-1}]=\rfield[x,f_1,f_2][x^{-1}]$.
Thus it is sufficient to extend $\{x,f_1,f_2,N_{\mathcal M}(z)\}$ to a SAGBI basis.
We use the {\it SAGBI /Divide-by-$x$} algorithm of \cite[\S1]{CSW} to do this.
We will show that the algorithm produces one new invariant, which we denote by $f_3$, and
that $\LT(f_3)=\gamma_{1357}y^{p^3+2}$.
For $p=3$ and $p=5$, this result follows from a Magma calculation. 
For the rest of this section, we assume $p>5$.

Expanding the definitions of $f_1$, $f_{12346}$ and $f_2$ gives
\[f_1=\gamma_{1234}y^{p^2}+\gamma_{1235}\delta^px^{p^2-2p}+\gamma_{1245}x^{p^2-p}y^p+\gamma_{1345}\delta x^{p^2-2}+\gamma_{2345}x^{p^2-1}y,\]
\[f_{12346}=-\gamma_{1234}\delta^{p^2}+\gamma_{1236}\delta^px^{2p^2-2p}+\gamma_{1246}x^{2p^2-p}y^p+\gamma_{1346}\delta x^{2p^2-2}+\gamma_{2346}x^{2p^2-1}y\]
and
\[\begin{split}
f_2&=\frac{f_1^2+\gamma_{1234}f_{12346}}{2x^{p^2-2p}}
 =\gamma_{1234}\gamma_{1235}y^{p^2}\delta^p+\gamma_{1234}\gamma_{1245}x^py^{p^2+p}+\gamma_{1234}\gamma_{1345}\delta x^{2p-2}y^{p^2}\\
&+\gamma_{1234}\gamma_{2345}x^{2p-1}y^{p^2+1}+\frac{\gamma_{1234}^2}{2}x^{2p}z^{p^2}+\frac{\gamma_{1235}^2}{2}\delta^{2p}x^{p^2-2p}+\gamma_{1235}\gamma_{1245}\delta^px^{p^2-p}y^p\\
&+\gamma_{1235}\gamma_{1345}\delta^{p+1}x^{p^2-2}+\gamma_{1235}\gamma_{2345}\delta^px^{p^2-1}y+\frac{\gamma_{1234}\gamma_{1236}}{2}x^{p^2}\delta^p+\frac{\gamma_{1245}^2}{2}x^{p^2}y^{2p}\\
&+\gamma_{1245}\gamma_{1345}\delta x^{p^2+p-2}y^p+\gamma_{1245}\gamma_{2345}x^{p^2+p-1}y^{p+1}+\frac{\gamma_{1234}\gamma_{1246}}{2}y^px^{p^2+p}\\
&+\frac{\gamma_{1345}^2}{2}\delta^2x^{p^2+2p-4}+\gamma_{1345}\gamma_{2345}\delta x^{p^2+2p-3}y+\frac{\gamma_{2345}^2}{2}x^{p^2+2p-2}y^2\\
&+\frac{\gamma_{1234}\gamma_{1346}}{2}\delta x^{p^2+2p-2}+\frac{\gamma_{1234}\gamma_{2346}}{2}x^{p^2+2p-1}y.
\end{split}\]
Subducting the \tat\ $(f_1^{p+2},f_2^p)$ gives
\[\begin{split}
 \widetilde{f_3}&=\underbrace{\gamma_{1235}^pf_1^{p+2}}_{T_1}-\underbrace{\gamma_{1234}^2f_2^p}_{T_2}+\underbrace{\alpha_1x^{p^2-2p}f_1^pf_2}_{T_3}\\&+\underbrace{\alpha_2x^{p^2}f_1^{p+1}}_{T_4}+\underbrace{ \alpha_3x^{2p^2-2p}f_1^{p-1}f_2}_{T_5}+\underbrace{\alpha_4x^{2p^2-p}f_1^{\frac{p-3}{2}}f_2^{\frac{p+1}{2}}}_{T_6}\end{split}\]
where
\[\alpha_1=-2\gamma_{1235}^p,\quad\alpha_2=\gamma_{1234}\gamma_{1245}^p,\quad \alpha_3=\frac{\gamma_{1234}^{p+1}\gamma_{1237}}{\gamma_{1235}}\text{ and }\alpha_4=\frac{\gamma_{1234}^{p+3}\gamma_{1257}}{\gamma_{1235}^{\frac{p+3}{2}}}.\]
%for $p>3$. For $p=3$, $\alpha_1$ is as above, 
%\[\alpha_2=\gamma_{1234}\gamma_{1245}^p+\frac{\gamma_{1235}^6}{\gamma_{1234}^2},\quad \alpha_3=\frac{\gamma_{1234}^{p+1}\gamma_{1237}}{\gamma_{1235}}-\frac{\gamma_{1235}^6}{\gamma_{1234}^2}
%\text{ and }\alpha_4=\frac{\gamma_{1234}^{6}\gamma_{1257}}{\gamma_{1235}^{3}}-\alpha_3.\]

\begin{lemma}\label{lem:genf3}
 For $p\geq 5$, $\LT(\widetilde{f_3})=\alpha x^{2p^2-2}y^{p^3+2}$ with 
 \[\alpha=\frac{\gamma_{1234}^{p+1}}{\gamma_{1235}}(\gamma_{1234}\gamma_{1345}^{p+1}+\gamma_{1235}^p\gamma_{1345}\gamma_{1236}
  -\gamma_{1235}^{p+1}\gamma_{1346})=-\frac{\gamma_{1357}\gamma_{1234}^{2p+2}}{\gamma_{1235}}\]
\end{lemma}

\begin{proof} 
 We work modulo the ideal in $\rfield[x,y,z]$ generated by $x^{2p^2-1}$.
  By the definition of $f_2$, we have
\[T_1-T_2+T_3= -\gamma_{1235}^p\gamma_{1234}f_1^pf_{12346}-\gamma_{1234}^2f_2^p.\]
As $f_1^p\equiv \gamma_{1234}^py^{p^3}$ and 
\[f_2^p\equiv \gamma_{1234}^p\gamma_{1235}^p\delta^{p^2}y^{p^3}+\gamma_{1234}^p\gamma_{1245}^px^{p^2}y^{p^3+p^2}+\gamma_{1234}^p\gamma_{1345}^p\delta^px^{2p^2-2p}y^{p^3}+\gamma_{1234}^p\gamma_{2345}^px^{2p^2-p}y^{p^3+p},\]
we obtain
\[\begin{split}
T_1-T_2+T_3&\equiv-\gamma_{1234}^{p+2}\gamma_{1245}^px^{p^2}y^{p^3+p^2}-\gamma_{1234}^{p+1}(\gamma_{1234}\gamma_{1345}^p+\gamma_{1235}^p\gamma_{1236})\delta^px^{2p^2-2p}y^{p^3}\\&-\gamma_{1234}^{p+1}(\gamma_{1234}\gamma_{2345}^p+\gamma_{1235}^p\gamma_{1246})x^{2p^2-p}y^{p^3+p}-\gamma_{1234}^{p+1}\gamma_{1235}^p\gamma_{1346}\delta x^{2p^2-2}y^{p^3}.\end{split}\]
%We consider now $T_4$, i.e. $x^{p^2}f_1^{p+1}$ :
Since
\[\begin{split}
x^{p^2}f_1^{p+1}&\equiv \gamma_{1234}^py^{p^3}x^{p^2}f_1
\equiv
\gamma_{1234}^{p+1}x^{p^2}y^{p^3+p^2}
+\gamma_{1234}^{p}\gamma_{1235}\delta^px^{2p^2-2p}y^{p^3}\\
&+\gamma_{1234}^{p}\gamma_{1245}x^{2p^2-p}y^{p^3+p}
+\gamma_{1234}^{p}\gamma_{1345}\delta x^{2p^2-2}y^{p^3},
%
% %\gamma_{1234}^{p+2}\gamma_{1245}^px^{p^2}y^{p^3+p^2}+\gamma_{1234}^{p+1}\gamma_{1235}\gamma_{1245}^p\delta^px^{2p^2-2p}y^{p^3}\\&+\gamma_{1234}^{p+1}\gamma_{1245}^{p+1}x^{2p^2-%p}y^{p^3+p}+\gamma_{1234}^{p+1}\gamma_{1245}^p\gamma_{1345}\delta %x^{2p^2-2}y^{p^3},
\end{split}\]
we see that
\[\begin{split}
T_1-T_2+T_3+T_4
 &\equiv \gamma_{1234}^{p+1}(\gamma_{1235}\gamma_{1245}^p-\gamma_{1235}^p\gamma_{1236}-\gamma_{1234}\gamma_{1345}^p)x^{2p^2-2p}y^{p^3}\delta^p\\&+\gamma_{1234}^{p+1}(\gamma_{1245}^{p+1}-\gamma_{1235}^p\gamma_{1246}-\gamma_{1234}\gamma_{2345}^p)x^{2p^2-p}y^{p^3+p}\\
 &+\gamma_{1234}^{p+1}(\gamma_{1245}^p\gamma_{1345}-\gamma_{1235}^p\gamma_{1346})\delta x^{2p^2-2}y^{p^3}.\end{split}\]

Using Lemma~\ref{lem:first-plucker} for $i=3$ and $i=4$, along with the analogous result coming from the
$(1,3,4)(3,4,5,6,7)$ Pl\"ucker relation for $\Gamma$, gives
%Using Plücker relations :
%\begin{itemize}
%\item for (1,2,3),(3,4,5,6,7), we have $\gamma_{1235}\gamma_{1245}^p-\gamma_{1235}^p\gamma_{1236}-\gamma_{1234}\gamma_{1345}^p=-\gamma_{1237}\gamma_{1234}^p$
%\item for (1,2,4),(3,4,5,6,7), we have $\gamma_{1245}^{p+1}-\gamma_{1235}^p\gamma_{1246}-\gamma_{1234}\gamma_{2345}^p=-\gamma_{1247}\gamma_{1234}^p$
%\item for (1,3,4),(3,4,5,6,7), we have $\gamma_{1345}\gamma_{1245}^p-\gamma_{1346}\gamma_{1235}^p=-\gamma_{1347}\gamma_{1234}^p$
%\end{itemize}
\[T_1-T_2+T_3+T_4\equiv -\gamma_{1234}^{2p+1}\gamma_{1237}x^{2p^2-2p}y^{p^3}\delta^p-\gamma_{1234}^{2p+1}\gamma_{1247}x^{2p^2-p}y^{p^3+p}
 -\gamma_{1234}^{2p+1}\gamma_{1347}\delta x^{2p^2-2}y^{p^3}.\]
%Write $f_1=\gamma_{1234}y^{p^2}+x^{p^2-2p}P$ and expand to get
%\[f_1^{p-1}=\sum_{n=0}^{p-1}\binom{p-1}{n}\gamma_{1234}^{p-1-n}y^{(p-1-n)p^2}x^{np^2-2np}P^n.\]
%The last terms of the sum ($n>2$) vanish (because $p>5$) so we obtain 
Since $3p^2-4p\geq 2p^2-1$ for $p\geq 5$,
$x^{2p^2-2p}f_1^{p-1}\equiv \gamma_{1234}^{p-1}y^{p^3-p^2}x^{2p^2-2p}$.
Using the description of $f_2$ given above
\[x^{2p^2-2p}f_2\equiv 
       \gamma_{1234}x^{2p^2-2p}y^{p^2}\left(\gamma_{1235}\delta^p+\gamma_{1245}x^py^p+\gamma_{1345}\delta x^{2p-2}\right).\]
 Thus
\[T_5 \equiv 
    \alpha_3\gamma_{1234}^py^{p^3}x^{2p^2-2p}\left(\gamma_{1235}\delta^p+\gamma_{1245}x^py^p+\gamma_{1345}\delta x^{2p-2}\right).\]
Using the $(1,2,4)(1,2,3,5,7)$ and $(1,3,5)(1,2,3,4,7)$ Pl\"ucker relations gives 
\[T_1-T_2+T_3+T_4+T_5
 \equiv -\frac{\gamma_{1234}^{2p+2}\gamma_{1257}}{\gamma_{1235}}x^{2p^2-p}y^{p^3+p}
-\frac{\gamma_{1234}^{2p+2}\gamma_{1357}}{\gamma_{1235}}\delta x^{2p^2-2}y^{p^3}.\]
Expanding and reducing modulo $\langle x^{2p^2-1}\rangle$, we get 
\[x^{2p^2-p}f_1^{\frac{p-3}{2}}\equiv x^{2p^2-p}\gamma_{1234}^{\frac{p-3}{2}}y^{\frac{p^3-3p^2}{2}}\]
and
%\[f_2=\gamma_{1234}\gamma_{1235}y^{p^2+2p}+x^pQ\]
%and modulo our ideal, we conserve only the first term : 
\[x^{2p^2-p}f_2^{\frac{p+1}{2}}\equiv\gamma_{1234}^{\frac{p+1}{2}}\gamma_{1235}^{\frac{p+1}{2}}x^{2p^2-p}y^{\frac{p^3+3p^2}{2}+p}.\]
Thus
\[\frac{T_6}{\alpha_4}\equiv\gamma_{1234}^{p-1}\gamma_{1235}^{\frac{p+1}{2}}x^{2p^2-p}y^{p^3+p}\]
and
\[\widetilde{f_3}=T_1-T_2+T_3+T_4+T_5+T_6\equiv\alpha x^{2p^2-2}y^{p^3+2}.\]
Using the (1,2,3)(1,3,4,5,6) and (1,3,5)(3,4,5,6,7) Pl\"ucker relations, we obtain 
\[\alpha=\frac{\gamma_{1234}^{p+2}}{\gamma_{1235}}(\gamma_{1345}^{p+1}-\gamma_{1356}\gamma_{1235}^p)=-\frac{\gamma_{1234}^{2p+2}\gamma_{1357}}{\gamma_{1235}}\]
and, since we are using the grevlex term order with $x<y<z$, the result follows.
\end{proof}

Define 
$$f_3:=-\widetilde{f}_3\frac{\gamma_{1235}}{\gamma_{1234}^{2p+2}x^{2p^2-2}}$$ so that $\LT(f_3)=\gamma_{1357}y^{p^3+2}$. 
Looking at the exponents of $y$ modulo $p$, it is clear that there
is only one new non-trivial \tat\ : $(f_3^p,f_2f_1^{p^2-1})$. 
To prove that $\mathcal B:=\{x,f_1,f_2,f_3,N_{\mathcal M}(z)\}$ is SAGBI basis for $\rfield[V_{\mathcal M}]^E$, it is sufficient to show that
this \tat\ subducts to zero. However, $N_{\mathcal M}(z)$ is rather complicated and it is more conveniant to take an indirect approach.
Subducting the \tat\ using only $\{x,f_1,f_2,f_3\}$ gives
\[\widetilde{f_4}:=\underbrace{\beta_1 f_3^p}_{T'_1}-\underbrace{\beta_2 f_1^{p^2-1}f_2}_{T'_2}
+\underbrace{\beta_3 x^pf_1^{p^2-\frac{p+3}{2}}f_2^{\frac{p+1}{2}}}_{T'_3}
+\underbrace{\beta_4 x^{2p-2}f_1^{p^2-p}f_3}_{T'_4}+\underbrace{\beta_5 x^{2p-1}f_1^{\frac{p^2-1}{2}-p}f_2^{\frac{p-1}{2}}f_3^{\frac{p+1}{2}}}_{T'_5}\]
where
\[\beta_1:=\gamma_{1235}\gamma_{1234}^{p^2},\quad \beta_2:=\gamma_{1357}^p,
\quad \beta_3:=\frac{\gamma_{1234}(\gamma_{1245}\gamma_{1357}^p-\gamma_{1235}\gamma_{2357}^p)}{\gamma_{1235}^{(p+1)/2}},\]
\[\beta_4:=\gamma_{1234}^p\gamma_{1345}\gamma_{1357}^{p-1} \quad \text{ and }
\quad \beta_5:=-\gamma_{1234}^{\frac{p^2+p+2}{2}}\gamma_{1235}^{\frac{p+3}{2}}\gamma_{1357}^{\frac{p-3}{2}}.\]

The lemma below proves that $\{x,f_1,f_2,f_3,\widetilde{f_4}/x^{2p}\}$ is SAGBI basis. We then use this in the proof of
Theorem~\ref{thm:generic}.

\begin{lemma}\label{lem:genf4} For $p\geq 5$, $\LT(\widetilde{f_4})=\dfrac{\gamma_{1234}^{p^2}\gamma_{1235}^{p+1}}{2}x^{2p}z^{p^4}$.
% $\LM(\widetilde{f_4})=x^{2p}z^{p^4}$.
\end{lemma}

\begin{proof} We work modulo the ideal in $\rfield[x,y,z]$ generated by  $x^{2p+1}$ and $x^{2p}y$, 
which we denote by $\mathfrak{n}$. Since $p\geq 5$, we have $p^2-2p\geq 2p+1$. Therefore,
using the expressions for $f_1$ and $f_2$ given above
$p\geq5$, we have  $f_1\equiv_{\mathfrak{n}} \gamma_{1234}y^{p^2}$ and
\[\begin{split}
f_2&\equiv_{\mathfrak{n}} \gamma_{1234}\gamma_{1235}y^{p^2}\delta^p+\gamma_{1234}\gamma_{1245}y^{p^2+p}x^p+\gamma_{1234}\gamma_{1345}\delta x^{2p-2}y^{p^2}\\&+\gamma_{1234}\gamma_{2345}x^{2p-1}y^{p^2+1}+\frac{\gamma_{1234}^2}{2}x^{2p}z^{p^2}.\end{split}\]
We will need expressions modulo $\mathfrak{n}$ for $f_3^p$, $x^{2p-2}f_3$ and $x^{2p-1}f_3^{(p+1)/2}$.
Let $\mathfrak{m}$ denote the ideal generated by $x^2y$ and $x^3$. Reworking the calulations of the proof of Lemma~\ref{lem:genf3}
to keep additional terms of $f_3$, gives
$$f_3
 \equiv_{\mathfrak m} \gamma_{1357}\delta y^{p^3}+\gamma_{2357}xy^{p^3+1}+\frac{\gamma_{1235}}{2}x^2z^{p^3}.$$
Thus
\[f_3^p\equiv_{\mathfrak n} \gamma_{1357}^p\delta^p y^{p^4}+\gamma^p_{2357}x^py^{p^4+p}+\frac{\gamma^p_{1235}}{2}x^{2p}z^{p^4},\]
%
%-\frac{\gamma_{1234}^{2p^2+2p}\gamma_{1357}^p}{\gamma_{1235}^p}\delta^py^{p^4}
%-\frac{\gamma_{1234}^{2p^2+2p}\gamma_{2357}^p}{\gamma_{1235}^p}x^py^{p^4+p}-\frac{\gamma_{1234}^{2p^2+2p}}{2^p}x^{2p}z^{p^4},\]
%
\[x^{2p-2}f_3\equiv_{\mathfrak n}  \gamma_{1357}\delta x^{2p-2} y^{p^3}+\gamma_{2357}x^{2p-1}y^{p^3+1}+\frac{\gamma_{1235}}{2}x^{2p}z^{p^3}\]
%-\frac{\gamma_{1234}^{2p+2}\gamma_{1357}}{\gamma_{1235}}x^{2p-2}\delta y^{p^3}-
%\frac{\gamma_{1234}^{2p+2}\gamma_{2357}}{\gamma_{1235}}x^{2p-1}y^{p^3+1}-\frac{\gamma_{1234}^{2p+2}}{2}x^{2p}z^{p^3}\]
and
\[x^{2p-1}f_3^{\frac{p+1}{2}}\equiv_{\mathfrak n} \gamma_{1357}^{(p+1)/2}x^{2p-1}y^{(p^3+2)(p+1)/2}.\]
% (-1)^{\frac{p+1}{2}}\frac{\gamma_{1234}^{p^2+2p+1}\gamma_{1357}^{\frac{p+1}{2}}}{\gamma_{1235}^{\frac{p+1}{2}}}
% x^{2p-1}y^{(p^3+2)(p+1)/2}.\]
Therefore
\[\begin{split}
T'_1-T'_2
&\equiv_{\mathfrak n} \gamma_{1234}^{p^2}\left(\gamma_{1235}\gamma_{2357}^p-\gamma_{1245}\gamma_{1357}^p\right)x^py^{p^4+p}
-\gamma_{1234}^{p^2}\gamma_{1345}\gamma_{1357}^p\delta x^{2p-2}y^{p^4}\\
&-\gamma_{1234}^{p^2}\gamma_{2345}\gamma_{1357}^px^{2p-1}y^{p^4+1}+\frac{\gamma_{1234}^{p^2}\gamma_{1235}^{p+1}}{2}x^{2p}z^{p^4}\end{split}\]

Since $x^pf_2^{\frac{p+1}{2}}\equiv_{\mathfrak n}\gamma_{1234}^{\frac{p+1}{2}}\gamma_{1235}^{\frac{p+1}{2}}x^py^{\frac{p^3+3p^2}{2}+p}$, we have 
\[\begin{split}T'_1-T'_2+T'_3
 &\equiv_{\mathfrak n} -\gamma_{1234}^{p^2}\gamma_{1345}\gamma_{1357}^p\delta x^{2p-2}y^{p^4}\\
&-\gamma_{1234}^{p^2}\gamma_{2345}\gamma_{1357}^px^{2p-1}y^{p^4+1}+\frac{\gamma_{1234}^{p^2}\gamma_{1235}^{p+1}}{2}x^{2p}z^{p^4}.\end{split}\]
Using the description of $x^{2p-2}f_3$ given above, we see that
\[T'_1-T'_2+T'_3+T'_4\equiv_{\mathfrak n} \gamma_{1234}^{p^2}\gamma_{1357}^{p-1}\left(\gamma_{1345}\gamma_{2357}-\gamma_{1357}\gamma_{2345}\right)x^{2p-1}y^{p^4+1}
+\frac{\gamma_{1234}^{p^2}\gamma_{1235}^{p+1}}{2}x^{2p}z^{p^4}.\]
The $(2,3,5)(1,3,4,5,7)$ Pl\"ucker relation gives 
$\gamma_{2345}\gamma_{1357}-\gamma_{2357}\gamma_{1345}=-\gamma_{1235}\gamma_{3457}$. Thus
\[T'_1-T'_2+T'_3+T'_4\equiv_{\mathfrak n} \gamma_{1234}^{p^2}\gamma_{1357}^{p-1}\gamma_{1235}\gamma_{3457}x^{2p-1}y^{p^4+1}
+\frac{\gamma_{1234}^{p^2}\gamma_{1235}^{p+1}}{2}x^{2p}z^{p^4}.\]
Observe that 
\[x^{2p-1}f_1^{\frac{p^2-1}{2}-p}\equiv_{\mathfrak n} \gamma_{1234}^{\frac{p^2-1}{2}-p}x^{2p-1}y^{\frac{p^4-p^2}{2}-p^3}\]
and
\[x^{2p-1}f_2^{\frac{p-1}{2}}\equiv_{\mathfrak n} \gamma_{1234}^{\frac{p-1}{2}}\gamma_{1235}^{\frac{p-1}{2}}x^{2p-1}y^{\frac{p^3+p^2}{2}-p}.\]
Therefore, using the description of $x^{2p-1}f_3^{\frac{p+1}{2}}$ given above, we obtain
\[\widetilde{f_4}:=T'_1-T'_2+T'_3+T'_4+T'_5\equiv_{\mathfrak n} \frac{\gamma_{1234}^{p^2}\gamma_{1235}^{p+1}}{2}x^{2p}z^{p^4},\]
and, since we are using the grevlex term order with $x<y<z$, the result follows.
\end{proof}

\begin{theorem}\label{thm:generic}
The set $\mathcal B:=\{x,f_1,f_2,f_3,N_{\mathcal M}(z)\}$ is a SAGBI basis, and hence a generating set,
for $\rfield[V_{\mathcal M}]^E$.
Furthermore, $\rfield[V_{\mathcal M}]^E$ is a complete intersection with generating relations coming from the subduction of
the \tat s $(f_2^p,f_1^{p+2})$ and $(f_3^p,f_2f_1^{p^2-1})$. 
\end{theorem}
\begin{proof} Define $f_4:=\widetilde{f_4}/x^{2p}$,
  $\mathcal B':=\{x,f_1,f_2,f_3,f_4\}$ and let $A$ denote the algebra generated by
$\mathcal B'$. The only non-trivial \tat s for $\mathcal B'$ are
$(f_2^p,f_1^{p+2})$ and $(f_3^p,f_2f_1^{p^2-1})$. From Lemmas~\ref{lem:genf3} and \ref{lem:genf4}, these \tat s subduct to zero.
Therefore $\mathcal B'$ is a SAGBI basis for $A$. From Theorem~5.2 of \cite{CSW}, 
$\rfield[V_{\mathcal M}]^E[x^{-1}]=\rfield[x,f_1,f_2][x^{-1}]$. Thus $A[x^{-1}]=\rfield[V_{\mathcal M}]^E[x^{-1}]$. 
Note that $\LM(f_4)=z^{p^4}$. Therefore, by Theorem~\ref{thm:compute}, $A=\rfield[V_{\mathcal M}]^E$ and
$\mathcal B'$ is a SAGBI basis for $\rfield[V_{\mathcal M}]^E$. Hence the lead term algebra of $\rfield[V_{\mathcal M}]^E$
is generated by $\{x,y^{p^2},y^{p^2+2p},y^{p^3+2},z^{p^4}\}$. Since the orbit of $z$ has size $p^4$, we see that
 $\LM(N_{\mathcal M}(z))=z^{p^4}$. Thus $\LM(\mathcal B)=\LM(\mathcal B')$ and $\mathcal B$ is also a SAGBI basis for
$\rfield[V_{\mathcal M}]^E$. For any subalgebra with a SAGBI basis, the relations are generated by the non-trivial \tat.
Hence $(f_2^p,f_1^{p+2})$ and $(f_3^p,f_2f_1^{p^2-1})$ generate the ideal of relations and $\rfield[V_{\mathcal M}]^E$ 
is a complete intersection with embedding dimension five.
\end{proof}

\section{The Essentially Generic Case}
\label{sec:1357!=0,1234=!0,1235!=0}
In this section we consider representations $V_M$ for $M\in\field^{2\times 4}$
for which $\gamma_{1234}(M)\not=0$, $\gamma_{1235}(M)\not=0$ and  $\gamma_{1357}(M)\not=0$.
With this restriction on $M$, we can evaluate the coefficients of the polynomials
$\{f_i\mid i=1,2,3,4\}$, as defined in Section~\ref{sec:genericcase}, at $M$ to get 
$\{\bar{f_i}\mid i=1,2,3,4\}\subset\field[V_M]^E$. 
Note that $\LT(\bar{f_1})=\gamma_{1234}(M)y^{p^2}$ so that $\LM(\bar{f_1})=y^{p^2}$.
Similarly $\LM(\bar{f_2})=y^{p^2+2p}$, $\LM(\bar{f_3})=y^{p^3+2}$ and $\LM(\bar{f_4})=z^{p^4}$.
Also, note that $\gamma_{1357}(M)=0$ if and only if $\{c_{11},c_{12},c_{13},c_{14}\}$ is
linearly dependent over $\field_p$. Thus, if $\gamma_{1357}(M)\not=0$, the orbit of $z$ has
size $p^4$ and $\LM(N_M(z))=z^{p^4}$.

\begin{theorem}\label{thm:essgen}
If $\gamma_{1234}(M)\not=0$, $\gamma_{1235}(M)\not=0$ and  $\gamma_{1357}(M)\not=0$, 
then the set $\mathcal B:=\{x,\bar{f_1},\bar{f_2},\bar{f_3},N_{M}(z)\}$ is a SAGBI basis, and hence a generating set,
for $\field[V_{M}]^E$.
Furthermore, $\field[V_{M}]^E$ is a complete intersection with generating relations coming from the subduction of
the \tat s $(\bar{f_2}^p,\bar{f_1}^{p+2})$ and $(\bar{f_3}^p,\bar{f_2}\bar{f_1}^{p^2-1})$. 
\end{theorem}
\begin{proof}
Define $\mathcal B':=\{x,\bar{f_1},\bar{f_2},\bar{f_3},\bar{f_4}\}$ and let $A$ denote the algebra generated by
$\mathcal B'$. The only non-trivial \tat s for $\mathcal B'$ are $(\bar{f_2}^p,\bar{f_1}^{p+2})$ and 
$(\bar{f_3}^p,\bar{f_2}\bar{f_1}^{p^2-1})$. 
The calculations in the proofs of Lemmas~\ref{lem:genf3} and \ref{lem:genf4} survive evaluation at $M$, 
proving that these \tat s subduct to zero and $\mathcal B'$ is a SAGBI basis for $A$. 
Thus, to use Theorem~\ref{thm:compute} to prove $A=\field[V_{M}]^E$, we need only show that
$A[x^{-1}]=\field[V_{M}]^E[x^{-1}]$.
 
Consider 
$$f_{12357}=\gamma_{1235}y^{p^3}-\gamma_{1237}y^{p^2}x^{p^3-p^2}+\gamma_{1257}y^px^{p^3-p}+\gamma_{1357}\delta x^{p^3-2}+\gamma_{2357}yx^{p^3-1}$$
and evaluate the coefficients at $M$ to get $\bar{f}_{12357}\in\field[V_M]^E$ with lead monomial $y^{p^3}$.
Since $\gamma_{1357}(M)\not=0$, $\bar{f}_{12357}$ has degree one as a polynomial in $z$. Furthermore, the coefficient
of $z$ is $-\gamma_{1357}(M)x^{p^3-1}$. Therefore, using Theorem~2.4 of \cite{CC}, 
$\field[V_M]^E[x^{-1}]=\field[x,N_M(y),\bar{f}_{12357}][x^{-1}]$. 
Thus, to prove $A=\field[V_M]^E$, it is sufficient to show that $\{N_M(y),\bar{f}_{12357}\}\subset A[x^{-1}]$.

Using Lemma~\ref{lem:field} for the subsequence $(1,2,4)$ of $(1,2,3,4,5,7)$, shows that
$$\gamma_{1235}^p\tilde{f}_{12357}
=\gamma_{3457}\tilde{f}_{12357}\in\text{Span}_{\field_p}\{\gamma_{2357}\tilde{f}_{13457},\gamma_{1357}\tilde{f}_{23457}\}.$$
Thus $\bar{f}_{12357}\in\text{Span}_{\field[x,x^{-1}]}\{\bar{f}_{13457},\bar{f}_{23457}\}$.
Similarly, using the $(1,6,7)$ subsequence of $(1,3,4,5,6,7)$,
 $\bar{f}_{13457}\in\text{Span}_{\field[x,x^{-1}]}\{\bar{f}_{13456},\bar{f}_{12345}^p\}$.
Iterating this process gives 
$\bar{f}_{12357}\in\text{Span}_{\field[x,x^{-1}]}\{\bar{f}_{12345},\bar{f}_{12345}^p,\bar{f}_{12346}\}$.
Since $\bar{f}_{12345}=\bar{f_1}$ and $\bar{f}_{12346}=2\bar{f}_2x^{p^2-2p}-\bar{f}_1^2$,
we see that $\bar{f}_{13457}\in A[x^{-1}]$. A similar argument shows that
$\bar{f}_{13579}\in\text{Span}_{\field[x,x^{-1}]}\{\bar{f}^{p^i}_{12345},\bar{f}^{p^j}_{12346}\mid i,j\in\{0,1,2\}\}$,
giving $\bar{f}_{13579}\in A[x^{-1}]$.
Since $\bar{f}_{13579}=\gamma_{1357}(M)N_M(y)$ (see Remark~\ref{rem:norm}), we have $N_M(y)\in A[x^{-1}]$. 
Therefore $A=\field[V_M]^E$.
As in the proof of Theorem~\ref{thm:generic}, observe that $\LM(\mathcal{B})=\LM(\mathcal B')$.
\end{proof}

\begin{remark}
Lemmas~\ref{lem:genf3} and \ref{lem:genf4} are only valid for $p>5$. However, for the Magma calculations
used to verify Theorem~\ref{thm:generic} for $p=3$ and $p=5$, only $\gamma_{1234}$ and $\gamma_{1235}$ are inverted.
Thus Theorem~\ref{thm:essgen} remains valid for $p=3$ and $p=5$.
\end{remark}

\section{The $\gamma_{1234}=0$,$\gamma_{1235}\not=0$,$\gamma_{1357}\not=0$ Stratum}
\label{sec:1357!=0,1234=0,1235!=0}
In this section we consider representations $V_M$ for $M\in\field^{2\times 4}$
for which $\gamma_{1234}(M)=0$, $\gamma_{1235}(M)\not=0$ and  $\gamma_{1357}(M)\not=0$.
For convenience, we write $\bg_{ijk\ell}$ for $\gamma_{ijk\ell}(M)$.
Evaluating coefficients gives
$$\bar{f_1}=\bg_{1235}\delta^px^{p^2-2p}+\bg_{1245}y^px^{p^2-p}+\bg_{1345}\delta x^{p^2-2}+\bg_{2345}yx^{p^2-1}.$$
Define
$$h_1:=\frac{\bar{f_1}}{\bg_{1235} x^{p^2-2p}} \text{ and } h_2:=\frac{\bar{f}_{12357}}{\bg_{1235}}$$
so that $\LT(h_1)=y^{2p}$ and $\LT(h_2)=y^{p^3}$. Note that $h_1,h_2\in\field[V_M]^E$. Furthermore,
arguing as in the proof of Theorem~\ref{thm:essgen}, $\field[V_M]^E[x^{-1}]=\field[x,N_M(y),h_2][x^{-1}]$.

\begin{lemma}
$$N_M(y)=h_2^p+\left(\frac{\bg^p_{1237}}{\bg^p_{1235}}-\frac{\bg_{1359}}{\bg_{1357}}\right)h_2x^{p^4-p^3}
-\frac{\bg^p_{1357}}{\bg^p_{1235}}h_1x^{p^4-2p}.$$
\end{lemma}
\begin{proof}
Since $\bar{f}_{13579}=\bg_{1357}N_M(y)$ (see Remark~\ref{rem:norm}), we have
$$N_M(y)=y^{p^4}-\frac{\bg_{1359}}{\bg_{1357}}y^{p^3}x^{p^4-p^3}+\frac{\bg_{1379}}{\bg_{1357}}y^{p^2}x^{p^4-p^2}
 -\frac{\bg_{1579}}{\bg_{1357}}y^px^{p^4-p}+\frac{\bg_{3579}}{\bg_{1357}}yx^{p^4-1}.$$
Using the definition gives
$$h_2:=y^{p^3}-\frac{\bg_{1237}}{\bg_{1235}}y^{p^2}x^{p^3-p^2}+\frac{\bg_{1257}}{\bg_{1235}}y^{p}x^{p^3-p}
 +\frac{\bg_{1357}}{\bg_{1235}}\delta x^{p^3-2}+\frac{\bg_{2357}}{\bg_{1235}}yx^{p^3-1}.$$
Thus
\[\begin{split}
N_M(y)-h_2^p&=\left(\frac{\bg^p_{1237}}{\bg^p_{1235}}-\frac{\bg_{1359}}{\bg_{1357}}\right)y^{p^3}x^{p^4-p^3}
-\left(\frac{\bg^p_{1257}}{\bg^p_{1235}}-\frac{\bg_{1379}}{\bg_{1357}}\right)y^{p^2}x^{p^4-p^2}\\
&-\left(\frac{\bg^p_{2357}}{\bg^p_{1235}}+\frac{\bg_{1579}}{\bg_{1357}}\right)y^px^{p^4-p}
-\frac{\bg^p_{1357}}{\bg^p_{1235}}\delta^p x^{p^4-2p}+\frac{\bg_{3579}}{\bg_{1357}}yx^{p^4-1}.
\end{split}\]
Using the $(1,3,5)(3,4,5,7,9)$, $(1,3,7)(3,4,5,7,9)$ and $(1,5,7)(3,4,5,7,9)$ Pl\"ucker relations gives
\[\begin{split}
N_M(y)-h_2^p&=\frac{\bg^{p-1}_{1357}}{\bg^p_{1235}}\left(
\bg_{1345}y^{p^3}x^{p^4-p^3}
-\bg_{1347}y^{p^2}x^{p^4-p^2}
-\bg_{1457}y^px^{p^4-p}
-\bg_{1357}\delta^p x^{p^4-2p}\right)\\
&+\frac{\bg_{3579}}{\bg_{1357}}yx^{p^4-1}.
\end{split}
\]
Using the $(1,2,3)(1,3,4,5,7)$ and $(1,2,5)(1,3,4,5,7)$ Pl\"ucker relations
$$\bg_{1347}=\frac{\bg_{1237}\bg_{1345}}{\bg_{1235}} \text{ and }
\bg_{1457}=\frac{\bg_{1245}\bg_{1357}-\bg_{1257}\bg_{1345}}{\bg_{1235}}.$$
Thus
\[\begin{split}
N_M(y)&=h_2^p+\frac{\bg^{p-1}_{1357}}{\bg^p_{1235}}\left(\bg_{1345}h_2x^{p^4-p^3}
-\frac{\bg_{1357}\bg_{1245}}{\bg_{1235}}y^px^{p^4-p}
-\frac{\bg_{1357}\bg_{1345}}{\bg_{1235}}\delta x^{p^4-2}\right)
\\
&
-\frac{\bg^p_{1357}}{\bg^p_{1235}}\delta^p x^{p^4-2p}
+\left(\frac{\bg_{3579}}{\bg_{1357}}
-\frac{\bg_{1357}^{p-1}\bg_{1345}\bg_{2357}}{\bg^{p+1}_{1235}}\right)yx^{p^4-1}.
\end{split}
\]
Using the $(1,3,5)(2,3,4,5,7)$ Pl\"ucker relation,
$\bg_{1345}\bg_{2357}=\bg_{1357}\bg_{2345}+\bg_{1235}^{p+1}$, giving
$$
\frac{\bg_{1357}^{p-1}\bg_{1345}\bg_{2357}}{\bg^{p+1}_{1235}}
=\frac{\bg_{2345}\bg_{1357}^p}{\bg_{1235}^p}+\bg_{1357}^{p-1}.$$
From the definition of $h_1$,
$$N_M(y)=h_2^p+\frac{\bg^{p-1}_{1357}\bg_{1345}}{\bg^p_{1235}}h_2x^{p^4-p^3}
-\frac{\bg^p_{1357}}{\bg^p_{1235}}h_1x^{p^4-2p}
+\left(\frac{\bg_{3579}}{\bg_{1357}}-\bg^{p-1}_{1357}\right)yx^{p^4-1}.$$
The result follows from the fact that $\bg_{3579}=\bg_{1357}^p$.
\end{proof}

As a consequence of the lemma, $\field[V_M]^E[x^{-1}]=\field[x,h_1,h_2][x^{-1}]$.
Thus applying the SAGBI/Divide-by-$x$ algorithm to $\{x,h_1,h_2,N_M(z)\}$ produces a generating set
for $\field[V_M]^E$. Subducting the \tat\ $(h_2^2,h_1^{p^2})$ gives
$$\widetilde{h_3}:=h_2^2-h_1^{p^2}+2\frac{\bg_{1237}}{\bg_{1235}}h_1^{p(p+1)/2}x^{p^3-p^2}
-2\frac{\bg_{1257}}{\bg_{1235}}h_1^{(p^2+1)/2}x^{p^3-p}.$$

\begin{lemma} \label{lem:h3}
$$\LT(\widetilde{h_3})=\frac{2\bg_{1357}}{\bg_{1235}}y^{p^3+2}x^{p^3-2}.$$
\end{lemma}
\begin{proof}
We work modulo the ideal in $\field[x,y,z]$ generated by $x^{p^3-1}$.
Therefore $h_1^{p^2}\equiv y^{2p^3}$, $h_1x^{p^3-p}\equiv y^{2p}x^{p^3-p}$
and
$$h_2^2\equiv y^{2p^3}-2\frac{\bg_{1237}}{\bg_{1235}}y^{p^3+p^2}x^{p^3-p^2}+2\frac{\bg_{1257}}{\bg_{1235}}y^{p^3+p}x^{p^3-p}
 +2\frac{\bg_{1357}}{\bg_{1235}}\delta y^{p^3} x^{p^3-2}.$$
Since $x^{p^3-p^2}h_1^p\equiv x^{p^3-p^2}y^{2p^2}$, we have $(h_1^p)^{(p+1)/2}x^{p^3-p^2}\equiv  x^{p^3-p^2} y^{p^3+p^2}$.
Thus
$$h_2^2\equiv h_1^{p^2}-2\frac{\bg_{1237}}{\bg_{1235}}h_1^{p(p+1)/2}x^{p^3-p^2}+2\frac{\bg_{1257}}{\bg_{1235}}h_1^{(p^2+1)/2}x^{p^3-p}
 +2\frac{\bg_{1357}}{\bg_{1235}}\delta y^{p^3} x^{p^3-2}.$$
Hence $\widetilde{h_3}\equiv 2\frac{\bg_{1357}}{\bg_{1235}}\delta y^{p^3} x^{p^3-2}$, and the result follows
\end{proof}

Define %$h_3:=\frac{\bg_{1235}}{2\bg_{1357}x^{p^3-2}}\widetilde{h_3}$ 
$h_3:=\bg_{1235}\widetilde{h_3}/(2\bg_{1357}x^{p^3-2})$
so that
$\LT(h_3)=y^{p^3+2}$. Subducting the \tat\ $(h_3^p,h_2^ph_1)$  gives
$$\widetilde{h_4}:=h_3^p-h_1h_2^p-\alpha_1x^ph_1^{(p^3+1)/2}
+\alpha_2x^{2p-2}h_3h_1^{(p^3-p^2)/2}-\alpha_3x^{2p-1}h_1^{(p^2-1)/2}h_2^{(p-3)/2}h_3^{(p+1)/2}$$
with $$\alpha_1:=\left(\frac{\bg_{2357}}{\bg_{1357}}\right)^p-\frac{\bg_{1245}}{\bg_{1235}},\
\alpha_2:=\frac{\bg_{1345}}{\bg_{1235}}\ {\rm and}\ 
\alpha_3:=\alpha_2\frac{\bg_{2357}}{\bg_{1357}}-\frac{\bg_{2345}}{\bg_{1235}}.$$
\begin{lemma}\label{lem:h4}
$$\LT(\widetilde{h_4})=\left(\frac{\bg_{1235}}{\bg_{1357}}\right)^px^{2p}z^{p^4}.$$
\end{lemma}
\begin{proof}
We work modulo the ideal $\mathfrak n:=\langle x^{2p+1},x^{2p}y\rangle$.
Using the definition of $h_3$ and methods analogous to the proof of Lemma~\ref{lem:h3}, it is not hard to show that
$$h_3\equiv_{\langle x^3,x^2y\rangle} \delta y^{p^3}+\frac{\bg_{2357}}{\bg_{1357}}xy^{p^3+1}+\frac{\bg_{1235}}{\bg_{1357}}x^2z^{p^3}.$$
Thus 
$$h_3^p\equiv_{\mathfrak n} \delta^p y^{p^4}+\left(\frac{\bg_{2357}}{\bg_{1357}}\right)^px^py^{p^4+p}
        +\left(\frac{\bg_{1235}}{\bg_{1357}}\right)^px^{2p}z^{p^4}.$$
Since $h_2\equiv_{\mathfrak n} y^{p^3}$, we have
$$h_1h_2^p\equiv_{\mathfrak n} y^{p^4}\left(\delta^p+\frac{\bg_{1245}}{\bg_{1235}}y^px^p
       +\frac{\bg_{1345}}{\bg_{1235}}\delta x^{2p-2}+\frac{\bg_{2345}}{\bg_{1235}}yx^{2p-1}\right).$$
Furthermore, since $x^ph_1\equiv_{\mathfrak n}x^p\delta^p$, expanding gives 
$x^ph_1^{(p^3+1)/2} \equiv_{\mathfrak n}x^py^{p^4+p}$. Therefore
$$h_3^p-h_1h_2^p-\alpha_1x^ph_1^{(p^3+1)/2}  \equiv_{\mathfrak n}
-\frac{\bg_{1345}}{\bg_{1235}}\delta x^{2p-2}y^{p^4}-\frac{\bg_{2345}}{\bg_{1235}}x^{2p-1}y^{p^4+1}
+\left(\frac{\bg_{1235}}{\bg_{1357}}\right)^px^{2p}z^{p^4}.
$$
Note that $x^{2p-2}h_1^{(p^3-p^2)/2)} \equiv_{\mathfrak n}x^{2p-2}y^{p^4-p^3}$. Thus
$$x^{2p-2}h_3h_1^{(p^3-p^2)/2)} \equiv_{\mathfrak n} x^{2p-2}\left(y^{p^4}\delta+\frac{\bg_{2357}}{\bg_{1357}}xy^{p^4+1}\right).$$
Hence
$$h_3^p-h_1h_2^p-\alpha_1x^ph_1^{(p^3+1)/2}+\alpha_2 x^{2p-2}h_3h_1^{(p^3-p^2)/2)} \equiv_{\mathfrak n}
\alpha_3 x^{2p-1}y^{p^4+1}+\left(\frac{\bg_{1235}}{\bg_{1357}}\right)^px^{2p}z^{p^4}.$$
Since
$x^{2p-1}h_1^{(p^2-1)/2}h_2^{(p-3)/2}h_3^{(p+1)/2} \equiv_{\mathfrak n} x^{2p-1}y^{p^4+1}$,
%we see that $\widetilde{h_4}\equiv_{\mathfrak n}\left(\frac{\bg_{1235}}{\bg_{1357}}\right)^px^{2p}z^{p^4}$,
the result follows
\end{proof}

Define $h_4:=\bg_{1357}\widetilde{h_4}/(\bg_{1357}x^{2p})$ so that $\LT(h_4)=z^{p^4}$.

\begin{theorem}
If $\gamma_{1234}(M)=0$, $\gamma_{1235}(M)\not=0$ and  $\gamma_{1357}(M)\not=0$, 
then the set $\mathcal B:=\{x,h_1,h_2,h_3,N_{M}(z)\}$ is a SAGBI basis, and hence a generating set,
for $\field[V_{M}]^E$.
Furthermore, $\field[V_{M}]^E$ is a complete intersection with generating relations coming from the subduction of
the \tat s $(h_2^2,h_1^{p^2})$ and $(h_3^p,h_1h_2^p)$. 
\end{theorem}

\begin{proof}
Define $\mathcal B':=\{x,h_1,h_2,h_3,h_4\}$ and let $A$ denote the algebra generated by
$\mathcal B'$. The only non-trivial \tat s for $\mathcal B'$ are
$(h_2^2,h_1^{p^2})$ and $(h_3^p,h_1h_2^p)$. Using
Lemmas~\ref{lem:h3} and \ref{lem:h4},
these \tat s subduct to zero, proving that $\mathcal B'$ is a SAGBI basis for $A$. 
Since  $\field[V_M]^E[x^{-1}]=\field[x,h_1,h_2][x^{-1}]$, using Theorem~\ref{thm:compute}, 
$A=\field[V_{M}]^E$. Finally, observe that $\LM(\mathcal{B})=\LM(\mathcal B')$.
\end{proof}

\section{The $\gamma_{1234}\not=0$,$\gamma_{1235}=0$,$\gamma_{1357}\not=0$ Strata}
\label{sec:1357!=0,1234!=0,1235=0}
In this section we consider representations $V_M$ for $M\in\field^{2\times 4}$
for which $\gamma_{1235}(M)=0$, $\gamma_{1234}(M)\not=0$ and  $\gamma_{1357}(M)\not=0$.
For convenience, we write $\bg_{ijk\ell}$ for $\gamma_{ijk\ell}(M)$.

\begin{lemma}\label{lem:coef-sec5} If $\bg_{1234}\not=0$, $\bg_{1235}=0$, and $\bg_{1357}\not=0$, then
$\bg_{1345}\not=0$.
\end{lemma}
\begin{proof}
Let $r_i$ denote row $i$ of the matrix $\Gamma(M)$. 
Since $\bg_{1234}\not=0$, the set $\{r_1,r_2,r_3,r_4\}$ is linearly independent.  
Using this and the hypothesis that $\bg_{1235}=0$, 
we conclude that $r_5$ is linear
combination of $\{r_1,r_2,r_3\}$, say $r_5=a_1r_1+a_2r_2+a_3r_3$. 
Since $r_3$ is non-zero and the entries of $r_5$ are the $p^{\rm th}$ powers of the entries of $r_3$,
 we see that $r_5$ is non-zero.
Suppose, by way of contradiction, that $\bg_{1345}=0$. Then $r_5$ is a non-zero linear
combination of $\{r_1,r_3,r_4\}$, say $r_5=b_1r_1+b_3r_3+b_4r_4$. Thus
$b_1r_1+b_3r_3+b_4r_4=a_1r_1+a_2r_2+a_3r_3$. Since  $\{r_1,r_2,r_3,r_4\}$ 
is linearly independent, $b_4=a_2=0$, $a_1=b_1$, $a_3=b_3$ and $r_5=a_1r_1+a_3r_3$, contradicting the assumption that
$\bg_{1357}\not=0$. 
\end{proof}

Take $f_1$ as defined in Section~\ref{sec:genericcase}, evaluate coefficients and divide by
$\bg_{1234}$ to get
$$\widehat{f_1}:=y^{p^2}+\frac{\bg_{1245}}{\bg_{1234}}y^px^{p^2-p}+\frac{\bg_{1345}}{\bg_{1234}}\delta x^{p^2-2}
           +\frac{\bg_{2345}}{\bg_{1234}}y x^{p^2-1}.$$
Note that $\widehat{f_1}$ is degree one in $z$ with coefficient 
$x^{p^2-2}\bg_{1345}/\bg_{1234}$ and so, using Theorem~2.4 of \cite{CC}, 
$\field[V_M]^E[x^{-1}]=\field[x,N_M(y),\widehat{f_1}][x^{-1}]$. 
%We could take $f_2$ as defined in Section~\ref{sec:genericcase}, evaluate coefficients and divide by
%$\bg_{1234}\bg_{1245}x^p$ to get an invariant with lead term $y^{p^2+p}$. However, it is more efficient to 
%construct a new invariant by subducting $N_M(y)$. 
Define
$$\widetilde{h_2}:=N_M(y)-\widehat{f_1}^{p^2}+\alpha_1 \widehat{f_1}^p x^{p^4-p^3}+\alpha_2\widehat{f_1}^2 x^{p^4-2p^2}$$
with 
$$\alpha_1:=\frac{\bg_{1359}}{\bg_{1357}}+\frac{\bg_{1245}^{p^2}}{\bg_{1234}^{p^2}}
\ {\rm and} \
\alpha_2:=\frac{\bg_{1345}^{p^2}}{\bg_{1234}^{p^2}}.
$$
We work modulo the ideal $\mathfrak n:=\langle x^{p^4-p^2-1}\rangle$.
Since $\bg_{1357}N_M(y)=\bar{f}_{13579}$ (see Remark~\ref{rem:norm}), we have
$N_M(y)\equiv_{\mathfrak n} y^{p^4}-\frac{\bg_{1359}}{\bg_{1357}}y^{p^3} x^{p^4-p^3}$.
Therefore
$$N_M(y)-\widehat{f_1}^{p^2}\equiv_{\mathfrak n} 
       - \left(\frac{\bg_{1359}}{\bg_{1357}}+\frac{\bg_{1245}^{p^2}}{\bg_{1234}^{p^2}}\right)y^{p^3}x^{p^4-p^3}
       -\frac{\bg_{1345}^{p^2}}{\bg_{1234}^{p^2}}\delta^{p^2}x^{p^4-2p^2}.$$ 
Thus 
$$N_M(y)-\widehat{f_1}^{p^2} +\alpha_1\widehat{f_1}^px^{p^4-p^3}\equiv_{\mathfrak n} 
-\frac{\bg_{1345}^{p^2}}{\bg_{1234}^{p^2}}\delta^{p^2}x^{p^4-2p^2}
\equiv_{\mathfrak n} -\frac{\bg_{1345}^{p^2}}{\bg_{1234}^{p^2}}y^{2p^2}x^{p^4-2p^2}.$$
Hence 
\begin{eqnarray*}
\widetilde{h_2}&=&N_M(y)-\widehat{f_1}^{p^2} +\alpha_1\widehat{f_1}^px^{p^4-p^3} +\alpha_2\widehat{f_1}^2x^{p^4-2p^2}\\
&\equiv_{\mathfrak n}& \frac{2\alpha_2}{\bg_{1234}}\left(\bg_{1245}y^{p^2+p}x^{p^4-p^2-p}+\bg_{1345}y^{p^2+2}x^{p^4-p^2-2}\right)
\end{eqnarray*}

We first consider the case $\bg_{1245}\not=0$.  Define 
$h_2:=\bg_{1234}^{p^2+1}\widetilde{h_2}/(2x^{p^4-p^2-p}\bg_{1345}^{p^2}\bg_{1245})$
so that $\LT(h_2)=y^{p^2+p}$.
Since $N_M(y)\in\field[x,\widehat{f_1},h_2]$,
we have $\field[V_M]^E[x^{-1}]=\field[x,\widehat{f_1},h_2][x^{-1}]$.
Subducting the \tat\ $(h_2^p,\widehat{f_1}^{p+1})$ gives
$$\widetilde{h_3}:=\widehat{f_1}^{p+1}-h_2^p+\left(\frac{\bg_{1345}}{\bg_{1245}}\right)^p \widehat{f_1}^{p-2}h_2^2x^{p^2-2p}.$$

\begin{lemma}\label{lem:h3-sec5} 
$\LT(\widetilde{h_3})=2\left(\frac{\bg_{1345}}{\bg_{1245}}\right)^{p+1}y^{p^3+p+2}x^{p^2-p-2}$.
\end{lemma}

\begin{proof} We work modulo the ideal $\langle x^{p^2-p-1} \rangle$. Thus $\widehat{f_1}\equiv y^{p^2}$.
Reviewing the definition of $h_2$, we see that
$$h_2^p\equiv y^{p^3+p^2}+\left(\frac{\bg_{1345}}{\bg_{1245}}\right)^py^{p^3+2p}x^{p^2-2p}$$
and
$$h_2^2 x^{p^2-2p} \equiv  y^{2p^2+2p}x^{p^2-2p}+2\left(\frac{\bg_{1345}}{\bg_{1245}}\right)y^{2p^2+p+2}x^{p^2-p-2}.$$
Thus 
$$\widehat{f_1}^{p+1}-h_2^p+\left(\frac{\bg_{1345}}{\bg_{1245}}\right)^p\widehat{f_1}^{p-2}h_2^2x^{p^2-2p}
\equiv 2\left(\frac{\bg_{1345}}{\bg_{1245}}\right)^{p+1}y^{p^3+p+2}x^{p^2-p-2}$$
and the result follows.
\end{proof}

Define $h_3:= \bg_{1245}^{p+1}\widetilde{h_3}/(2\bg_{1345}^{p+1}x^{p^2-p-2})$ so that $\LT(h_3)=y^{p^3+p+2}$.

\begin{lemma}\label{lem:sub5} 
Subducting the \tat\ $(h_3^p,\widehat{f_1}^{p^2-1}h_2^2)$ gives an invariant with lead term
$-\bg_{1245}^p\bg_{1234}^{p^2}z^{p^4}x^{p^2+2p}/(4\bg_{1345}^{p^2+p})$.
\end{lemma}

\begin{proof} Modulo the ideal $\langle x^{p^2+2p+1},x^{p^2+2p}y\rangle$, the expression
\begin{eqnarray*}
h_3^p-\widehat{f_1}^{p^2-1}h_2^2&+&\beta_1 h_3 \widehat{f_1}^{p^2-p+1}x^{p-2}
       +\beta_2 h_2 \widehat{f_1}^{p^2} x^{p}
       +\beta_3 \widehat{f_1}^{p^2+1} x^{2p}\\
      &+&\beta_4 h_2^4\widehat{f_1}^{p^2-4}x^{p^2-2p}
       + \beta_5 h_3 h_2^2\widehat{f_1}^{p^2-p-2}x^{p^2-p-2}\\
       &+& \beta_6 h_3^2 \widehat{f_1}^{p^2-2p} x^{p^2-4}
       + \beta_7  h_2^2 \widehat{f_1}^{p^2-2} x^{p^2}
      + \beta_8  h_3 \widehat{f_1}^{p^2-p} x^{p^2+p-2}\\
      &+& \beta_9  h_2 \widehat{f_1}^{p^2-1} x^{p^2+p}
      + \beta_{10}  h_3h_2^{p-1} \widehat{f_1}^{p^2-2p} x^{p^2+2p-2}\\
      &+& \beta_{11} h_3^{(p+1)/2}h_2^{(p-3)/2}\widehat{f_1}^{(p^2+1)/2-p}x^{p^2+2p-1}
\end{eqnarray*}
with
$$\beta_1:=2\frac{\bg_{1345}}{\bg_{1245}},\hspace{5mm} 
\beta_2:=-\left(\frac{\bg_{1245}}{\bg_{1234}}\right)^{p+1}\left(\frac{\bg_{1234}}{\bg_{1345}}\right)^p,$$ 
$$\beta_3:=\frac{1}{2}\left(\frac{\bg_{1245,}}{\bg_{1345}}\right)^p\left(\left(\frac{\bg_{2345}}{\bg_{1345}}\right)^{p^2}
        -\left(\frac{\bg_{1245}}{\bg_{1345}}\right)^{p^2}\left(\frac{\bg_{1245}}{\bg_{1234}}\right)^p\right),$$
$$\beta_4:=-\frac{1}{2}\left(\frac{\bg_{1345}}{\bg_{1245}}\right)^p, \hspace{5mm}
\beta_5:=2\left(\frac{\bg_{1345}}{\bg_{1245}}\right)^{p+1},\hspace{2mm}
\beta_6:=-2\left(\frac{\bg_{1345}}{\bg_{1245}}\right)^{p+2}\hspace{2mm},$$
$$\beta_7:=
\frac{1}{2}\left(\frac{\bg_{1245}}{\bg_{1345}}\right)^p\left(\frac{\bg_{1234}}{\bg_{1345}}\right)^{p^2-p}\frac{\bg_{1359}}{\bg_{1357}},
 \hspace{5mm}
\beta_8:=-\frac{\bg_{1345}}{\bg_{1234}}\beta_7,$$
$$\beta_9:=-\frac{1}{2}\left(\frac{\bg_{1245}}{\bg_{1345}}\right)^p\left(\frac{\bg_{1234}}{\bg_{1345}}\right)^{p^2}
        \left(\frac{\bg_{1245}\bg_{1379}}{\bg_{1234}\bg_{1357}}+\frac{\bg_{1579}}{\bg_{1357}}\right),$$
$$ \beta_{10}:=\frac{1}{2}\left(\frac{\bg_{1245}}{\bg_{1345}}\right)^{p-1}\left(\frac{\bg_{1234}}{\bg_{1345}}\right)^{p^2}
         \frac{\bg_{1379}}{\bg_{1357}} \hspace{2mm}{\rm and }\hspace{2mm}
\beta_{11}:=\frac{1}{2}\left(\frac{\bg_{1245}}{\bg_{1345}}\right)^p\left(\frac{\bg_{1234}}{\bg_{1345}}\right)^{p^2}
         \frac{\bg_{3579}}{\bg_{1357}},$$
is congruent to $-\bg_{1245}^p\bg_{1234}^{p^2}z^{p^4}x^{p^2+2p}/(4\bg_{1345}^{p^2+p})$.
\end{proof}

\begin{theorem}
If $\gamma_{1234}(M)\not=0$, $\gamma_{1235}(M)=0$, $\gamma_{1357}(M)\not=0$, and $\gamma_{1245}(M)\not=0$, 
then the set $\mathcal B:=\{x,\widehat{f_1},h_2,h_3,N_{M}(z)\}$ is a SAGBI basis
 for $\field[V_{M}]^E$.
Furthermore, $\field[V_{M}]^E$ is a complete intersection with generating relations 
coming from the subduction of the \tat s $(h_2^p,\widehat{f_1}^{p+1})$ and 
$(h_3^p,\widehat{f_1}^{p^2-1}h_2^2)$. 
\end{theorem}

\begin{proof}
Use the subduction of  $(h_3^p,\widehat{f_1}^{p^2-1}h_2^2)$ given in Lemma~\ref{lem:sub5} to construct 
an invariant $h_4$ with lead term $z^{p^4}$. Define $\mathcal B':=\{x,\widehat{f_1},h_2,h_3,h_4\}$ and 
let $A$ denote the algebra generated by
$\mathcal B'$. The only non-trivial \tat s for $\mathcal B'$ are
 $(h_2^p,\widehat{f_1}^{p+1})$ and 
$(h_3^p,\widehat{f_1}^{p^2-1}h_2^2)$. 
Using Lemmas~\ref{lem:h3-sec5} and \ref{lem:sub5},
these \tat s subduct to zero, proving that $\mathcal B'$ is a SAGBI basis for $A$. 
Since  $\field[V_M]^E[x^{-1}]=\field[x,\widehat{f_1},h_2][x^{-1}]$, using Theorem~\ref{thm:compute}, 
$A=\field[V_{M}]^E$. Finally, observe that $\LM(\mathcal{B})=\LM(\mathcal B')$.
\end{proof}

We now consider the case $\bg_{1245}=0$. Define 
$\widehat{h_2}:=\bg_{1234}^{p^2+1}\widetilde{h_2}/(2x^{p^4-p^2-2}\bg_{1345}^{p^2+1})$
so that $\LT(\widehat{h_2})=y^{p^2+2}$.
Since $N_M(y)\in\field[x,\widehat{f_1},\widehat{h_2}]$,
we have $\field[V_M]^E[x^{-1}]=\field[x,\widehat{f_1},\widehat{h_2}][x^{-1}]$.

\begin{lemma}\label{lem:sub5v2}
Subducting the \tat\ $(\widehat{h_2}^{p^2},\widehat{f_1}^{p^2+2})$ gives an invariant with lead term 
$z^{p^4}\left(\bg_{1234}x^2/(2\bg_{1345})\right)^{p^2}$.
\end{lemma}
\begin{proof} Modulo the ideal $\langle x^{p^2+1},x^{p^2}y\rangle$,
the expression
\begin{eqnarray*}
\widehat{f_1}^{p^2+2}-\widehat{h_2}^{p^2}&-&\left(
       \alpha_1\widehat{h_2}\widehat{f_1}^{p^2}x^{p^2-2}
       +\alpha_2\widehat{f_1}^{p^2+1} x^{p^2}\right.\\
      &+&\alpha_3\widehat{h_2}^p\widehat{f_1}^{p^2-p} x^{2p^2-2p}
      +\alpha_4 \widehat{h_2}^{p(p+1)/2}\widehat{f_1}^{(p^2-p-2)/2}x^{2p^2-p}\\
     &+&\left.\alpha_5 \widehat{h_2}\widehat{f_1}^{p^2-1}x^{2p^2-2}
      +\alpha_6 \widehat{h_2}^{(p^2+1)/2}\widehat{f_1}^{(p^2-3)/2} x^{2p^2-1}\right)
\end{eqnarray*}
with
$$\alpha_1:=\frac{2\bg_{1345}}{\bg_{1234}},\hspace{5mm} 
\alpha_2:=-\frac{\bg_{1379}\bg_{1234}^{p^2}}{\bg_{1357}\bg_{1345}^{p^2}}, \hspace{5mm}
\alpha_3:=-\frac{\bg_{1359}\bg_{1234}^{p^2-p}}{\bg_{1357}\bg_{1345}^{p^2-p}},$$
$$\alpha_4:=\frac{\bg_{1579}\bg_{1234}^{p^2}}{\bg_{1357}\bg_{1345}^{p^2}}, \hspace{5mm}
\alpha_5:=\frac{\bg_{1379}\bg_{1234}^{p^2-1}}{\bg_{1357}\bg_{1345}^{p^2-1}}\hspace{2mm} {\rm and}\hspace{2mm}
\alpha_6:=-\frac{\bg_{3579}\bg_{1234}^{p^2}}{\bg_{1357}\bg_{1345}^{p^2}},$$
is congruent to $z^{p^4}\left(\bg_{1234}x^2/(2\bg_{1345})\right)^{p^2}$.
\end{proof}

\begin{theorem}
If $\gamma_{1234}(M)\not=0$, $\gamma_{1235}(M)=0$, $\gamma_{1357}(M)\not=0$, and $\gamma_{1245}(M)=0$,
then the set $\mathcal B:=\{x,\widehat{f_1},\widehat{h_2},N_{M}(z)\}$ is a SAGBI basis
 for $\field[V_{M}]^E$.
Furthermore, $\field[V_{M}]^E$ is a hypersurface with the relation 
coming from the subduction of the \tat\ $(\widehat{h_2}^{p^2},\widehat{f_1}^{p^2+2})$ 
\end{theorem}

\begin{proof}
Use the subduction of  $(\widehat{h_2}^{p^2},\widehat{f_1}^{p^2+2})$  given in
Lemma~\ref{lem:sub5v2} to construct 
an invariant $\widehat{h_3}$ with lead term $z^{p^4}$. Define $\mathcal 
B':=\{x,\widehat{f_1},\widehat{h_2},\widehat{h_3}\}$ and 
let $A$ denote the algebra generated by
$\mathcal B'$. The only non-trivial \tat\ for $\mathcal B'$ is 
$(\widehat{h_2}^{p^2},\widehat{f_1}^{p^2+2})$, which subducts to zero using
Lemma~\ref{lem:sub5v2}. Thus $\mathcal B'$ is a SAGBI basis for $A$. 
Since  $\field[V_M]^E[x^{-1}]=\field[x,\widehat{f_1},\widehat{h_2}][x^{-1}]$, using Theorem~\ref{thm:compute}, 
$A=\field[V_{M}]^E$. Finally, observe that $\LM(\mathcal{B})=\LM(\mathcal B')$.
\end{proof}

\section{The $\gamma_{1234}\not=0$,$\gamma_{1235}\not=0$,$\gamma_{1357}=0$ Stratum}
\label{sec:1357=0,1234!=0,1235!=0}

In this section we consider representations $V_M$ for $M\in\field^{2\times 4}$
for which $\gamma_{1234}(M)\not=0$, $\gamma_{1235}(M)\not=0$ and  $\gamma_{1357}(M)=0$.
For convenience, we write $\bg_{ijk\ell}$ for $\gamma_{ijk\ell}(M)$.
Evaluating the coefficients of $f_1$ and dividing by $\bg_{1234}$ gives 
$\widehat{f_1}$ with lead term $y^{p^2}$.
Since $\bg_{1357}=0$ and $\bg_{1235}\not=0$, the orbit of $y$ has size $p^3$ and
$N_M(y)=\bar{f}_{12357}/\bg_{1235}$ (see Remark~\ref{rem:norm}).
For convenience, write
$$N_M(y)=y^{p^3}+\alpha_2y^{p^2}x^{p^3-p^2}+\alpha_1y^px^{p^3-p}+\alpha_0yx^{p^3-1}$$
and
$$\widehat{f_1}=y^{p^2}+\beta_3\delta^px^{p^2-2p}+\beta_2y^px^{p^2-p}+\beta_1\delta x^{p^2-2}+\beta_0 yx^{p^2-1},$$
with $\alpha_2=-\bg_{1237}/\bg_{1235}$, $\alpha_1=\bg_{1257}/\bg_{1235}$, $\alpha_0=\bg_{2357}/\bg_{1235}$,
$\beta_3=\bg_{1235}/\bg_{1234}$, $\beta_2=\bg_{1245}/\bg_{1234}$,
$\beta_1=\bg_{1345}/\bg_{1234}$ and $\beta_0=\bg_{2345}/\bg_{1234}$.

Subducting $N_M(y)$ gives
$$\widetilde{h_2}:=N_M(y)-\widehat{f_1}^p+\beta_3^p x^{p^3-2p^2}\widehat{f_1}^2.$$

\begin{lemma} $\LT(\widetilde{h_2})=2\left(\frac{\bg_{1235}}{\bg_{1234}}\right)^{p+1}y^{p^2+2p}x^{p^3-p^2-2p}$.
\end{lemma}
\begin{proof} We work modulo the ideal $\langle x^{p^3-p^2-p}\rangle$. Using the definitions
of $f_{12357}$ and $f_{12345}$, we have $N_M(y)\equiv y^{p^3}$ and  
$\widehat{f_1}^p\equiv y^{p^3}+\left(\frac{\bg_{1235}}{\bg_{1234}}\right)^py^{2p^2}x^{p^3-2p^2}$.
The result follows from the observation that
$$\widehat{f_1}x^{p^3-2p^2}\equiv y^{p^2}x^{p^3-2p^2}+\left(\frac{\bg_{1235}}{\bg_{1234}}\right)y^{2p}x^{p^3-p^2-2p}.$$
\end{proof}
Define $h_2:=\widetilde{h_2}\bg_{1234}^{p+1}/(2\bg_{1235}^{p+1}x^{p^3-p^2-2p})$ so that $\LT(h_2)=y^{p^2+2p}$
and
\begin{eqnarray}\label{h2cong}
h_2&\equiv_{\langle x^{2p}\rangle}& y^{p^2}\left(\delta^p+\frac{\beta_2}{\beta_3}y^px^p+\frac{\beta_1}{\beta_3}\delta x^{2p-2}
        +\frac{\beta_0}{\beta_3}yx^{2p-1}\right).
\end{eqnarray}

\begin{lemma}\label{lem:field-sec6} $\field[V_M]^E[x^{-1}]=\field[x,\widehat{f_1},h_2][x^{-1}]$.
\end{lemma}
\begin{proof}
Since $\bg_{1357}=0$ and the first row of $M$ is non-zero, we can use a change of coordinates,
see \cite[\S4]{CSW}, and the ${\rm  GL}_4(\field_p)$-action to write
$$M=\begin{pmatrix} 
1&c_{12}&c_{13}&0\\
0&c_{22}&c_{23}&c_{24}
\end{pmatrix}.
$$
Since $\bg_{1235}\not=0$, we have $c_{24}\not=0$. With this choice of generators for $E$, let $H$ 
denote the subgroup generated by $e_1$ and $e_4$. Using the calculation of $\field[x,y,z]^H$ from
Theorem~6.4 of \cite{CSW}, we see that $\field[V_M]^H[x^{-1}]=\field[x,N_H(y),N_H(\delta)][x^{-1}]$
with $N_H(y):=y^p-yx^{p-1}$ and $N_H(\delta)=\delta^p-\delta (c_{24}x^2)^{p-1}$. 
Thus, to compute $\field[V_M]^G[x^{-1}]=\left(\field[V_M]^H[x^{-1}]\right)^{G/H}$, it is sufficient
to compute $$\left(\field[x,N_H(y),N_H(\delta)][x^{-1}]\right)^{G/H}
=\field[x,N_H(y)/x^{p-1},N_H(\delta)/x^{2p-1}]^{G/H}[x^{-1}].$$
Note that $\deg(N_H(y)/x^{p-1})=\deg(N_H(\delta)/x^{2p-1})=1$.
Furthermore $$\field[x,N_H(y)/x^{p-1}]^{G/H}=\field[x,N_{G/H}(N_H(y)/x^{p-1})]$$ and
$N_{G/H}(N_H(y)/x^{p-1})=N_M(y)/x^{p^3-p^2}$.
Using the form of $M$ given above, we see that $\bg_{1345}=-c_{24}^{p-1}\bg_{1235}$.
If we evaluate $\widetilde{\Gamma}$ at $M$ and set $x=1$, $y=1$, and $z=1$, then first
and last columns of the  resulting matrix are equal. Thus $\bar{f}_{12345}(1,1,1)=\bg_{1234}+\bg_{1245}+\bg_{2345}=0$.
Using these two relations, we can write
$$\widehat{f_1}=N_H(y)^p-\frac{\bg_{2345}}{\bg_{1234}}N_H(y)x^{p^2-p}+\frac{\bg_{1235}}{\bg_{1234}}N_H(\delta)x^{p^2-2p}.$$
Thus $\widehat{f_1}/x^{p^2-p}\in \field[x,N_H(y)/x^{p-1},N_H(\delta)/x^{2p-1}]^{G/H}$ is of degree $1$ 
in $N_H(\delta)/x^{2p-1}$ with coefficient $x^{p-1}\bg_{1235}/\bg_{1234}$. Thus by Theorem~2.4 of \cite{CC}, we have
$$\field[x,N_H(y)/x^{p-1},N_H(\delta)/x^{2p-1}]^{G/H}[x^{-1}]=\field[x,N_M(y)/x^{p^3-p^2},\widehat{f_1}/x^{p^2-p}][x^{-1}].$$
Therefore $\field[V_M]^E[x^{-1}]=\field[x,N_M(y),\widehat{f_1}][x^{-1}]$. The result then follows from the fact that
$N_M(y)\in \field[x,\widehat{f_1},h_2]$.
\end{proof}

Subducting the \tat\ $(h_2^p,\widehat{f_1}^{p+2})$ gives
\begin{eqnarray*}
\widetilde{h_3}&:=&h_2^p - \widehat{f_1}^{p+2}
                 +2\beta_3\widehat{f_1}^ph_2x^{p^2-2p}\\
                  && -\beta_3^{-p}\left(
                 \alpha_2\widehat{f_1}^{p+1}x^{p^2}
                 -\alpha_2\beta_3\widehat{f_1}^{p-1}h_2x^{2p^2-2p}
                 +\alpha_1\widehat{f_1}^{(p-3)/2}h_2^{(p+1)/2}x^{2p^2-p}\right)
\end{eqnarray*}
for $p\geq 5$ and
\begin{eqnarray*}
\widetilde{h_3}&:=&h_2^3 - \widehat{f_1}^5
                 +2\beta_3\widehat{f_1}^3h_2x^3\\
                && -\left(\alpha_2\beta_3^{-3}+\beta_3^3\right)\left(\widehat{f_1}^4x^9-\beta_3\widehat{f_1}^2h_2x^{12}\right)
                 -\left(\alpha_1\beta_3^{-3}+\alpha_2\beta_3^{-1}+\beta_3^5\right)h_2^2x^{15}
\end{eqnarray*}
for $p=3$.

\begin{lemma}\label{lem:h3-sec6} $\LT(\widetilde{h_3})=\alpha_0\beta_3^{-p}y^{p^3+1}x^{2p^2-1}$.
\end{lemma}
\begin{proof}
For $p=3$, this is a Magma calculation. Suppose $p\geq 5$.
We work modulo the ideal $\langle x^{2p^2}\rangle$. Since $p^3-2p^2>2p^2$,  we have 
$\widehat{f_1}^p\equiv y^{p^3}$. Furthermore, $3p^2-4p>2p^2$,
giving $\widehat{f_1}x^{2p^2-2p}\equiv y^{p^2}x^{2p^2-2}$.
Using Congruence~\ref{h2cong} given above, %the definition of $h_2$,
we have
$$h_2x^{2p^2-2p}\equiv x^{2p^2-2p}y^{p^2}\left(\delta^p+\frac{\beta_2}{\beta_3}y^px^p+\frac{\beta_1}{\beta_3}\delta x^{2p-2}
          +\frac{\beta_0}{\beta_3}yx^{2p-1}\right)$$
and
$$h_2^p\equiv y^{p^3}\left(
                        \delta^p+\frac{\beta_2}{\beta_3}y^px^p+\frac{\beta_1}{\beta_3}\delta x^{2p-2}
                        +\frac{\beta_0}{\beta_3}yx^{2p-1}
                        \right)^p.$$
Using the definition of $h_2$, we get
\begin{eqnarray*}
\widehat{f_1}^2-2\beta_3h_2x^{p^2-2p}&=&\beta_3^{-p}x^{2p^2-p^3}\left(\widehat{f_1}^p-N_M(y)\right)\\
  &=&\delta^{p^2} +\beta_3^{-p}\left(\left( \beta_2^p-\alpha_2\right)y^{p^2} x^{p^2}
                                 + \beta_1^p\delta^p x^{2p^2-2p}\right. \\
  & &                               +\left.\left(\beta_0^p-\alpha_1\right)y^px^{2p^2-p}
                                 -\alpha_0 yx^{2p^2-1}\right).
\end{eqnarray*}
Thus
$$h_2^p-\widehat{f_1}^p\left(\widehat{f_1}^2-2\beta_3h_2x^{p^2-2p}\right)\equiv
\frac{y^{p^3}}{\beta_3^p}\left(\alpha_2y^{p^2}x^{p^2}+\alpha_1y^px^{2p^2-p}+\alpha_0yx^{2p^2-1}\right).$$
Furthermore, using the above expressions,
$$\widehat{f_1}^{p+1}x^{p^2}-\beta_3\widehat{f_1}^{p-1}h_2x^{2p^2-2p}
\equiv y^{p^3-p^2}x^{p^2}\left(y^{p^2}\widehat{f_1}-\beta_3 h_2x^{p^2-2p}\right)\equiv x^{p^2}y^{p^3+p^2}.$$
Therefore
\begin{eqnarray*}
h_2^p-\widehat{f_1}^p\left(\widehat{f_1}^2-2\beta_3h_2x^{p^2-2p}\right)
        &-&\frac{\alpha_2}{\beta_3^p}\left(\widehat{f_1}^{p+1}x^{p^2}-\beta_3\widehat{f_1}^{p-1}h_2x^{2p^2-2p}\right)\\
       &\equiv& \frac{y^{p^3}}{\beta_3^p}\left(\alpha_1y^px^{2p^2-p}+\alpha_0yx^{2p^2-1}\right)
\end{eqnarray*}
Note that $h_2x^{2p^2-p}\equiv y^{p^2+2p}x^{2p^2-p}$ and $\widehat{f_1}x^{2p^2-p}\equiv y^{p^2}x^{2p^2-p}$. Hence
$$\widehat{f_1}^{(p-3)/2}h_2^{(p+1)/2}x^{2p^2-p}\equiv y^{p^3+p}x^{2p^2-p},$$ giving
$\widetilde{h_3}\equiv \alpha_0y^{p^3+1}x^{2p^2-1}/\beta_3^p$, as required.
\end{proof}

Note that $\alpha_0/\beta_3^p=\bg_{2357}\bg_{1234}^p/\bg_{1235}^{p+1}$.
Since $\bg_{1357}=0$, $\bg_{1235}\not=0$, and $\bg_{3457}=\bg_{1235}^p\not=0$, 
arguing as in the proof of Lemma~\ref{lem:coef-sec5},
we see that $\bg_{2357}\not=0$. Define $h_3:=\bg_{1235}^{p+1}\widetilde{h_3}/(x^{2p^2-1}\bg_{2357}\bg_{1234}^p)$ so that 
$\LT(h_3)=y^{p^3+1}$.

\begin{lemma}\label{lem:sub6} $\LM\left(h_3^p-h_2^{(p^2+1)/2}\widehat{f_1}^{(p^2-2p-1)/2}\right)=x^pz^{p^4}$.
\end{lemma}
\begin{proof}
Working modulo the ideal $\mathfrak{n}:=\langle x^{p+1},x^py\rangle$,  we see that
$\widehat{f_1}\equiv_{\mathfrak n} y^{p^2}$ and
$h_2\equiv_{\mathfrak n} y^{p^2+2p}$, giving
$h_3^p-h_2^{(p^2+1)/2}\widehat{f_1}^{(p^2-2p-1)/2}\equiv_{\mathfrak n} h_3^p-y^{p^4+p}$.
Thus it is sufficient to identify the lead monomial of $h_3-y^{p^3+1}$.
Note that $y^{p^3+1}$ and $xz^{p^3}$ are consecutive monomials in the grevlex term order.
Therefore, if $xz^{p^3}$ appears with non-zero coefficient in $h_3$, then $\LM(h_3-y^{p^3+1})=xz^{p^3}$,
and the result follows. Work modulo the ideal
 $\mathfrak{m}:=\langle y\rangle$. Then 
$\widehat{f_1}\equiv_{\mathfrak m} -\beta_3 z^px^{p^2-p}-\beta_1zx^{p^2-1}$ and $N_M(y)\equiv_{\mathfrak m}0$.
Therefore 
$$h_2\equiv_{\mathfrak m}\frac{1}{2\beta_3}\left(z^{p^2}x^{2p}+\frac{\beta_1^p}{\beta_3^p}z^px^{p^2+p}
                     +x^{p^2}\left(\beta_3z^p+\beta_1zx^{p-1}\right)^2\right).$$
Hence $h_3$ has degree $p^3$ as a polynomial in $z$, with leading coefficient $x/2\alpha_0$
%Therefore $$h_3^p-h_2^{(p^2+1)/2}\widehat{f_1}^{(p^2-2p-1)/2}\equiv_{\mathfrak n} \frac{x^pz^{p^4}}{2\alpha_0^p}
%                    =\frac{\bg_{1235}^p}{2\bg_{2357}^p}x^pz^{p^4},$$
and the result follows.
\end{proof}

\begin{theorem}
If $\gamma_{1234}(M)\not=0$, $\gamma_{1235}(M)\not=0$ and  $\gamma_{1357}(M)=0$, 
then the set $\mathcal B:=\{x,\widehat{f_1},h_2,h_3,N_{M}(z)\}$ is a SAGBI basis
for $\field[V_{M}]^E$. Furthermore, $\field[V_{M}]^E$ is a complete intersection 
with generating relations coming from the subduction of
the \tat s $(h_2^p,\bar{f_1}^{p+2})$ and $(h_3^p,\bar{f_1}^{(p^2-2p-1)/2}h_2^{(p^2+1)/2})$. 
\end{theorem}

\begin{proof}
Use the subduction given in Lemma~\ref{lem:sub6} to construct 
an invariant $h_4$ with lead term $z^{p^4}$. 
Define $\mathcal B':=\{x,\widehat{f_1},h_2,h_3,h_4\}$ and 
let $A$ denote the algebra generated by
$\mathcal B'$.
The only non-trivial \tat s for $\mathcal B'$ are
$(h_2^p,\bar{f_1}^{p+2})$ and $(h_3^p,\bar{f_1}^{(p^2-2p-1)/2}h_2^{(p^2+1)/2})$. Using
Lemmas~\ref{lem:h3-sec6} and \ref{lem:sub6},
these \tat s subduct to zero, proving that $\mathcal B'$ is a SAGBI basis for $A$. 
By Lemma~\ref{lem:field-sec6}, we have $\field[V_M]^E[x^{-1}]=\field[x,\widehat{f_1},h_2][x^{-1}]$. 
Using Theorem~\ref{thm:compute}, 
$A=\field[V_{M}]^E$. Clearly $\LT(N_M(z))=z^{p^k}$ for $k\leq 4$. Since $\mathcal{B'}$ is a SAGBI basis for
$\field[V_E]^E$, this forces $k=4$, giving
$\LM(\mathcal{B})=\LM(\mathcal B')$.
\end{proof}

\section{The $\gamma_{1234}=0$, $\gamma_{1235}=0$, $\gamma_{1357}\not=0$ Strata}
\label{sec:1357!=0,1234=0,1235=0}

In this section we consider representations $V_M$ for $M\in\field^{2\times 4}$
for which $\gamma_{1235}(M)=0$, $\gamma_{1234}(M)=0$ and  $\gamma_{1357}(M)\not=0$.
For convenience, we write $\bg_{ijk\ell}$ for $\gamma_{ijk\ell}(M)$.

We first consider the case $\bg_{1257}=0$. Let $r_i$ denote row $i$ of the matrix $\Gamma(M)$.
Since  $\gamma_{1357}(M)\not=0$,
$\{r_1,r_3,r_5,r_7\}$ is linearly independent. Thus $r_2$ is a linear combination
of $r_1$, $r_5$ and $r_7$. Since $\bg_{1235}=0$, $r_2$ is a linear combination of
 $r_1$, $r_3$ and $r_5$. 
Using the $(1,2,3)(3,4,5,7,9)$ Pl\"ucker relation, $\bg_{1237}=0$. Thus
 $r_2$ is a linear combination of
 $r_1$, $r_3$ and $r_7$. Combining these observations, we see that $r_2$ is
a scalar multiple of $r_1$. Using a change of coordinates, see \cite[\S 4]{CSW},
we may assume that $r_2$ is zero. 
If the second row of $M$ is zero, then $V_M$ is a symmetric square representation 
and the invariants are generated by $x$, $\delta$, $N_M(y)$ and $N_M(z)$.
Since $\bg_{1357}\not=0$, $N_M(y)$ and $N_M(z)$ are both of degree $p^4$ and there is a single relation in 
degree $2p^4$ which can be constructed by subducting the \tat\ $(\delta^{p^4},N_M(y)^2)$
 (see Theorem~3.3 of \cite{CSW}).

For the rest of this section, we assume $\bg_{1257}\not=0$. Evaluating coefficients gives the invariant
 $\bar{f}_{12357}$. Using the $(1,2,3)(3,4,5,7,9)$ Pl\"ucker relation, $\bg_{1237}^{p+1}=0$. Thus
$\bg_{1237}=0$, and we have
$\bar{f}_{12357}=\bg_{1257}y^px^{p^3-p}+\bg_{1357}\delta x^{p^3-2}+\bg_{2357}yx^{p^3-1}$. Divide by $\bg_{1257}x^{p^3-p}$ to get
$$h_1:=y^p+\frac{\bg_{1357}}{\bg_{1257}}\delta x^{p-2}+\frac{\bg_{2357}}{\bg_{1257}}y x^{p-1}.$$
Observe that $N_M(y)=\bar{f}_{13579}/\bg_{1357}$. Subducting $N_M(y)$ gives
\begin{eqnarray*}
\widetilde{h_2}=N_M(y)-h_1^{p^3}&+&\alpha^{p^3}h_1^{2p^2}x^{p^4-2p^3}
                                -2\alpha^{p^3+p^2}h_1^{p^2+2p}x^{p^4-p^3-2p^2}\\
                                &+&4\alpha^{p^3+p^2+p}h_1^{p^2+p+2}x^{p^4-p^3-p^2-2p}
\end{eqnarray*}
with $\alpha:=\bg_{1357}/\bg_{1257}$.

\begin{lemma}\label{lem:sec7-h2} $\LT(\widetilde{h_2})=8\alpha^{p^3+p^2+p+1}y^{p^3+p^2+p+2}x^{p^4-p^3-p^2-p-2}$.
\end{lemma}
\begin{proof}
It will be convenient to work modulo the ideal $\langle x^{p^4-p^3},x^{p^4-p^3-p^2-p-1}y\rangle$,
so that $N_M(y)\equiv y^{p^4}$ and $h_1^{p^3}\equiv y^{p^4}+\alpha^{p^3} \delta^{p^3} x^{p^4-2p^3}$.
Thus $N_M(y)-h_1^{p^3}\equiv -\alpha^{p^3} \delta^{p^3} x^{p^4-2p^3}$. Expanding gives
$$x^{p^4-2p^3}\left(h_1^{p^2}\right)^2\equiv 
      x^{p^4-2p^3}y^{p^3}\left(y^{p^3}+2\alpha^{p^2}\delta^{p^2}x^{p^3-2p^2}\right).$$
Thus 
$$N_M(y)-h_1^{p^3}+\alpha^{p^3}h_1^{2p^2}x^{p^4-2p^3} \equiv 2\alpha^{p^3+p^2}y^{p^3}\delta^{p^2}x^{p^4-p^3-2p^2}.$$
                                                      %+\alpha^{p^3}z^{p^3}x^{p^4-p^3}.$$
Again expanding gives
$$h_1^{p^2+2p}x^{p^4-p^3-2p^2}\equiv x^{p^4-p^3-2p^2} y^{p^3+p^2}\left(y^{p^2}+2\alpha^p\delta^px^{p^2-2p}\right).$$
Hence
\begin{eqnarray*}
N_M(y)-h_1^{p^3}+\alpha^{p^3}h_1^{2p^2}x^{p^4-2p^3}&-&2\alpha^{p^3+p^2}h_1^{p^2+2p}x^{p^4-p^3-2p^2}\\ 
&\equiv&
-4\alpha^{p^3+p^2+p}\delta^py^{p^3+p^2}x^{p^4-p^3-p^2-2p}.  
%+\alpha^{p^3}z^{p^3}x^{p^4-p^3}.
\end{eqnarray*}
Since $h_1^{p^2+p+2}x^{p^4-p^3-p^2-2p}\equiv x^{p^4-p^3-p^2-2p}y^{p^3+p^2+p}\left(y^p+2\alpha \delta x^{p-2}\right)$,
we have $$\widetilde{h_2}\equiv 8\alpha^{p^3+p^2+p+1}y^{p^3+p^2+p+2}x^{p^4-p^3-p^2-p-2}$$
%+\alpha^{p^3}z^{p^3}x^{p^4-p^3}$$
and the result follows.
\end{proof}

Define $h_2:=\widetilde{h_2}/(8\alpha^{p^3+p^2+p+1}x^{p^4-p^3-p^2-p-2})$ so that $\LT(h_2)=y^{p^3+p^2+p+2}$.

\begin{lemma} \label{lem:sub7}
Subducting the \tat\ $(h_2^p,h_1^{p^3+p^2+p+2})$ gives an invariant with lead term
$$\left(\frac{\bg_{1257}}{2\bg_{1357}}\right)^{p^3+p^2+p}z^{p^4}x^{p^3+p^2+2p}.$$
\end{lemma}
\begin{proof}
For $p=3$, this is a Magma calculation. For $p>3$, the subduction is given by
\begin{eqnarray*}
h_2^p &-& h_1^{p^3+p^2+p+2} + 2\alpha h_2h_1^{p^3}x^{p-2}\\
      &+&\frac{1}{4\alpha^{p^3+p^2+p}}\left(
               \beta_1 h_1^{p^3+p^2}x^{p^2+2p}-\beta_1\alpha^{p^2}h_1^{p^3+2p}x^{p^3-p^2+2p}\right.\\  
      & & \hspace{15mm}                      +2\beta_1\alpha^{p^2+p}h_1^{p^3+p+2}x^{p^3}
                            -4\beta_1\alpha^{p^2+p+1}h_2h_1^{p^3-p^2}x^{p^3+p-2} \\
      & & \hspace{15mm} - \beta_2x^{p^3}\left(  
                        h_1^{p^3+p}x^{2p}-\alpha^ph_1^{p^3+2}x^{p^2}
                        +2\alpha^{p+1}h_2h_1^{p^3-p^2-p}x^{p^2+p-2}\right)\\
       && \hspace{15mm} +\beta_3x^{p^3+p^2+p}\left(
                          h_1^{p^3+1}-\alpha h_2h_2^{p^3-p^2-p-1}x^{p-2}\right)\\
       &&\hspace{15mm} -\left. \beta_4h_2^{(p+1)/2}h_1^{(p^2+p+1)(p-3)/2}x^{p^3+p^2+2p-1}\right)
\end{eqnarray*}
with
$$\alpha:=\frac{\bg_{1357}}{\bg_{1257}}, \hspace{10mm} \beta_1:=\frac{\bg_{1359}}{\bg_{1357}},
\hspace{10mm} \beta_2:=\frac{\bg_{1379}}{\bg_{1357}},
\hspace{10mm} \beta_3:=\frac{\bg_{1579}}{\bg_{1357}}$$
and $\beta_4:=\bg_{1357}^{p-1}$.
To calculate the lead term, work modulo the ideal generated by $x^{p^3+p^2+2p+1}$ and $x^{p^3+p^2+2p}y$. 
\end{proof}

\begin{theorem}
If $\gamma_{1234}(M)=0$, $\gamma_{1235}(M)=0$, $\gamma_{1357}(M)=0$
and $\gamma_{1257}(M)\not=0$, 
then the set $\mathcal B:=\{x,h_1,h_2,N_{M}(z)\}$ is a SAGBI basis
for $\field[V_{M}]^E$. Furthermore, $\field[V_{M}]^E$ is a
hypersurface
with the relation coming from the subduction of
the \tat\ $(h_2^p,h_1^{p^3+p^2+p+2})$.
\end{theorem}
\begin{proof}
Use the subduction given in Lemma~\ref{lem:sub7} to construct 
an invariant $h_3$ with lead term $z^{p^4}$. 
Define $\mathcal B':=\{x,h_1,h_2,h_3\}$ and 
let $A$ denote the algebra generated by
$\mathcal B'$.
The only non-trivial \tat\ for $\mathcal B'$ 
is $(h_2^p,h_1^{p^3+p^2+p+2})$, which subducts to $0$ using the definition of $h_3$.
Thus  $\mathcal B'$ is a SAGBI basis for $A$.
Since $h_1$ is degree $1$ in $z$ with coefficient $-\alpha x^{p-1}$, it follows from
\cite{CC} that  $\field[V_M]^E[x^{-1}]=\field[x,h_1,N_M(y)][x^{-1}]$.
Since $N_M(y)\in\field[x,h_1,h_2]$, we have  $\field[V_M]^E[x^{-1}]=\field[x,h_1,h_2][x^{-1}]$.
Using Theorem~\ref{thm:compute}, 
$A=\field[V_{M}]^E$. Clearly $\LT(N_M(z))=z^{p^k}$ for $k\leq 4$. Since $\mathcal{B'}$ is a SAGBI basis for
$\field[V_E]^E$, this forces $k=4$, giving
$\LM(\mathcal{B})\subset\LM(\mathcal B')$.
\end{proof}

\section{The $\gamma_{1234}=0$, $\gamma_{1235}\not=0$, $\gamma_{1357}=0$ Stratum}
\label{sec:1357=0,1234=0,1235!=0}

In this section we consider representations $V_M$ 
with $\gamma_{1235}(M)\not=0$, $\gamma_{1234}(M)=0$ and  $\gamma_{1357}(M)=0$.
The results of this section are valid for $p\geq 3$.
For convenience, we write $\bg_{ijk\ell}$ for $\gamma_{ijk\ell}(M)$.
Observe that $N_M(y)=\bar{f}_{12357}/\bg_{1235}$ (see Remark~\ref{rem:norm}). 
Thus $N_M(y)$ has lead term $y^{p^3}$. Furthermore,
$\bar{f}_{12345}$ has lead term $\bg_{1235}y^{2p}x^{p^2-2p}$. 
Define $h_1:=\bar{f}_{12345}/(\bg_{1235}x^{p^2-2p})$ so that $\LT(h_1)=y^{2p}$.

\begin{lemma}\label{lem:field-sec8}
$\field[V_M]^E[x^{-1}]=\field[x,h_1,N_M(y)][x^{-1}]$.
\end{lemma}
\begin{proof}
We argue as in the proof of Theorem~4.4 of \cite{CSW}. Since 
$N_M(y)$ and $h_1/x^p$  are algebraically independent elements
 of $\field[x,y,\delta/x]^E$
with $\deg(N_M(y))\deg(h_1/x^p)=p^4=|E|$, applying Theorem~3.7.5 of \cite{DK} gives
$\field[x,y,\delta/x]^E=\field[x,N_M(y),h_1/x^p]$. The result then follows from the observation that
$$\field[x,y,z]^E[x^{-1}]=\field[x,y,\delta/x]^E[x^{-1}].$$
\end{proof}
Subducting the \tat\ $(N_M(y)^2,h_1^{p^2})$ gives
$$\widetilde{h_2}:=N_M(y)^2-h_1^{p^2} +\frac{2}{\bg_{1235}}
\left(\bg_{1237} x^{p^3-p^2}h_1^{(p^2+p)/2}-\bg_{1257} x^{p^3-p}h_1^{(p^2+1)/2}\right).$$

\begin{lemma}\label{lem:ny-sec8} $\LT(\widetilde{h_2})=2\bg_{2357}y^{p^3+1}x^{p^3-1}/\bg_{1235}$.
\end{lemma}
\begin{proof}
We work modulo the ideal $\langle x^{p^3} \rangle$. Expand $N_M(y)^2$ and observe that
$h_1^{p^2}\equiv y^{2p^3}$, $h_1^px^{p^3-p^2}\equiv y^{2p^2}x^{p^3-p^2}$ and $h_1x^{p^3-p}\equiv y^{2p}x^{p^3-p}$.
\end{proof}

Using the $(1,3,5)(2,3,4,5,7)$ Pl\"ucker relation, $\bg_{1345}\bg_{2357}=\bg_{1235}^{p+1}$. Thus
$\bg_{2357}\not=0$.
Define $h_2:=\bg_{1235}\widetilde{h_2}/(2\bg_{2357}x^{p^3-1})$ so that $\LT(h_2)=y^{p^3+1}$.

\begin{lemma}\label{lem:sub8} 
$\LM\left(h_2^p-h_1^{(p^3+1)/2)}\right)=z^{p^4}x^p$
\end{lemma}
\begin{proof}
A careful calculation shows that
$$\LT\left(h_2^p-h_1^{(p^3+1)/2}\right)=\frac{\bg_{1235}^p}{2\bg_{2357}^p}x^pz^{p^4}.$$
\end{proof}

\begin{theorem}
If $\gamma_{1234}(M)=0$, $\gamma_{1235}(M)\not=0$ and $\gamma_{1357}(M)=0$, 
then the set $\mathcal B:=\{x,h_1,h_2,N_M(y),N_{M}(z)\}$ is a SAGBI basis
for $\field[V_{M}]^E$. Furthermore, $\field[V_{M}]^E$ is a
complete intersection with relations coming from the subduction of the \tat s
 $(N_M(y)^2,h_1^{p^2})$  and $(h_2^p,h_1^{(p^3+1)/2})$
\end{theorem}

\begin{proof}
Use the subduction from Lemma~\ref{lem:sub8} to construct 
an invariant $h_3$ with lead term $z^{p^4}$. 
Define $\mathcal B':=\{x,N_M(y),h_1,h_2,h_3\}$ and 
let $A$ denote the algebra generated by
$\mathcal B'$.
The non-trivial \tat s for $\mathcal B'$
subduct to zero using Lemmas~\ref{lem:ny-sec8} and \ref{lem:sub8}.
Thus  $\mathcal B'$ is a SAGBI basis for $A$.
From Lemma~\ref{lem:field-sec8}, $\field[V_M]^E[x^{-1}]=\field[x,h_1,N_M(y)][x^{-1}]$.
Thus, using Theorem~\ref{thm:compute}, 
$A=\field[V_{M}]^E$. Clearly $\LT(N_M(z))=z^{p^k}$ for $k\leq 4$. Since $\mathcal{B'}$ is a SAGBI basis for
$\field[V_E]^E$, this forces $k=4$, giving
$\LM(\mathcal{B})=\LM(\mathcal B')$.
\end{proof}

\section{The $\gamma_{1234}\not=0$, $\gamma_{1235}=0$, $\gamma_{1357}=0$ Strata}
\label{sec:1357=0,1234!=0,1235=0}

In this section we consider representations $V_M$ 
for which $\gamma_{1235}(M)=0$, $\gamma_{1234}(M)\not=0$ and  $\gamma_{1357}(M)=0$.
For convenience, we write $\bg_{ijk\ell}$ for $\gamma_{ijk\ell}(M)$.
Using the $(1,3,5)(3,4,5,6,7)$ Pl\"ucker relation, $\bg_{1345}=0$.
Thus
$\bar{f_1}=\bg_{1234}y^{p^2}+\bg_{1245}y^px^{p^2-p}+\bg_{2345}yx^{p^2-1}\in\field[x,y]$.
Since $\bg_{1234}\not=0$, the orbit of $y$ contains at least $p^2$ elements.
Thus $N_M(y)=\bar{f_1}/\bg_{1234}$ (see Remark~\ref{rem:norm}).

\begin{lemma} \label{lem:field-sec9}
$\field[V_M]^E[x^{-1}]=\field[x,N_M(y),\bar{f}_{12346}][x^{-1}]$.
\end{lemma}
\begin{proof}
We argue as in the proof of Lemma~\ref{lem:field-sec8} (and Theorem~4.4 of \cite{CSW}). 
Since $N_M(y)$ and $\bar{f}_{12346}/x^{p^2}$ are algebraically independent elements
 of $\field[x,y,\delta/x]^E$
with $\deg(N_M(y))\deg(\bar{f}_{12346}/x^{p^2})=p^4=|E|$, applying Theorem~3.7.5 of \cite{DK} gives
$\field[x,y,\delta/x]^E=\field[x,N_M(y),\bar{f}_{12346}/x^{p^2}].$
The result then follows from the observation that
$\field[x,y,z]^E[x^{-1}]=\field[x,y,\delta/x]^E[x^{-1}]$.
\end{proof}

We first consider the case $\bg_{1245}\not=0$.
Define $\widehat{f_2}:=\bar{f_2}/(\bg_{1234}\bg_{1245}x^p)$ so that $\LT(\widehat{f_2})=y^{p^2+p}$.
Subduct the \tat\ $(\widehat{f_2}^p,N_M(y)^{p+1})$ to get
$$\widetilde{h_3}:=N_M(y)^{p+1}-\widehat{f_2}^p
-\left(\frac{\bg_{1245}}{\bg_{1234}}-\frac{\bg_{2345}^p}{\bg_{1245}^p}\right) \widehat{f_2}N_M(y)^{p-1}x^{p^2-p}.
$$

\begin{lemma}\label{lem:h3-sec9}
$$\LT(\widetilde{h_3})=\left(\frac{\bg_{2345}^{p+1}}{\bg_{1245}^{p+1}}\right) x^{p^2-1}y^{p^3+1}.$$
\end{lemma}
\begin{proof}
Expanding  and reducing modulo the ideal $\langle x^{p^2}\rangle$.
\end{proof}

Define $$h_3:=\frac{\bg_{1245}^{p+1}}{x^{p^2-1}\bg_{2345}^{p+1}}\widetilde{h_3}$$ so that $\LT(h_3)=y^{p^3+1}$.

\begin{lemma}\label{lem:sub9}
Subducting the \tat\ $(h_3^p,N_M(y)^{p^2-1}\widehat{f_2})$ gives an invariant with lead monomial
$x^pz^{p^4}$.
\end{lemma}
\begin{proof}
Work modulo the ideal $\langle x^{p+1},x^py\rangle$ and expand to get
$$h_3^p-\widehat{f_2}N_M(y)^{p^2-1}+\frac{\bg_{2345}}{\bg_{1234}}x^{p-1}h_3N_M(y)^{p^2-p}\equiv
\left(\frac{\bg_{1234}^{p^2}\bg_{1245}^p}{\bg_{2345}^{p^2+p}}\right)z^{p^4}x^p.$$
\end{proof}

\begin{theorem}
If $\gamma_{1234}(M)\not=0$, $\gamma_{1235}(M)=\gamma_{1357}(M)=0$, and $\gamma_{1245}(M)\not=0$, 
then the set $\mathcal B:=\{x,N_M(y),\widehat{f_2},h_3,N_{M}(z)\}$ is a SAGBI basis
for $\field[V_{M}]^E$. Furthermore, $\field[V_{M}]^E$ is a
complete intersection with relations coming from the subduction of the \tat s
 $(\widehat{f_2}^p,N_M(y)^{p+1})$  and $(h_3^p,N_M(y)^{p^2-1}\widehat{f_2})$.
\end{theorem}

\begin{proof}
Use the subduction given in Lemma~\ref{lem:sub9} to construct 
an invariant $h_4$ with lead term $z^{p^4}$. 
Define $\mathcal B':=\{x,N_M(y),\widehat{f_2},h_3,h_4\}$ and 
let $A$ denote the algebra generated by
$\mathcal B'$.
The non-trivial \tat s for $\mathcal B'$
subduct to $0$ using Lemmas~\ref{lem:h3-sec9} and \ref{lem:sub9}.
Thus  $\mathcal B'$ is a SAGBI basis for $A$.
From Lemma~\ref{lem:field-sec9}, $\field[V_M]^E[x^{-1}]=\field[x,N_M(y),\bar{f}_{12346}][x^{-1}]$.
However, since $f_2=(f_1^2+\gamma_{1234}f_{12346})/(2x^{p^2-2p})$, we see that
$$\field[x,N_M(y),\bar{f}_{12346}][x^{-1}]=\field[x,N_M(y),\widehat{f_2}][x^{-1}].$$
Thus, using Theorem~\ref{thm:compute}, 
$A=\field[V_{M}]^E$. Clearly $\LT(N_M(z))=z^{p^k}$ for $k\leq 4$. Since $\mathcal{B'}$ is a SAGBI basis for
$\field[V_E]^E$, this forces $k=4$, giving
$\LM(\mathcal{B})=\LM(\mathcal B')$.
\end{proof}

Suppose $\bg_{1245}=0$ and let $r_i$ denote row $i$ of the matrix $\Gamma(M)$.
Since $\bg_{1234}\not=0$, we see that $\{r_1,r_2,r_3,r_4\}$ is linearly independent.
Using the assumptions that $\bg_{1235}=\bg_{1245}=0$, we see that 
$r_5\in{\rm Span}(r_1,r_2,r_3)\cap{\rm Span}(r_1,r_2,r_4)$. Therefore
$r_5\in{\rm Span}(r_1,r_2)$. However, since $\bg_{1357}=0$, 
using a change of coordinates (see \cite[\S 4]{CSW}) and the ${\rm GL}_4(\field_p)$-action,
we may assume
$$M:=\begin{pmatrix} 
1&c_{12}&c_{13}&0\\
0&c_{22}&c_{23}&c_{24}
\end{pmatrix}
$$
with $c_{24}\not=0$. Since $r_5=r_1^{p^2}$, we conclude that $r_5=r_1$. Thus 
$\bg_{2345}=-\bg_{1234}$. Hence 
$N_M(y)=\bar{f_1}/\bg_{1234}=y^{p^2}-yx^{p^2-1}$.
Define $\widehat{h_2}:=-\bar{f_2}/\left(\bg_{1234}^2x^{2p-1}\right)$ so that
$\LT(\widehat{h_2})=y^{p^2+1}$.

\begin{theorem}
If $\gamma_{1234}(M)\not=0$ and $\gamma_{1235}(M)=\gamma_{1357}(M)=\gamma_{1245}(M)=0$, 
then the set $\mathcal B:=\{x,N_M(y),\widehat{h_2},N_{M}(z)\}$ is a SAGBI basis
for $\field[V_{M}]^E$. Furthermore, $\field[V_{M}]^E$ is a
hypersurface with the relation coming from the subduction of the \tat\
 $(\widehat{h_2}^{p^2},N_M(y)^{p^2+1})$.
\end{theorem}

\begin{proof}
Using the definition of $\widehat{h_2}$ and the description given above of $N_M(y)$,
we see that 
$\LT\left(\widehat{h_2}^{p^2}-N_M(y)^{p^2+1}-\widehat{h_2}\left(xN_M(y)\right)^{p^2-1}\right)
=-z^{p^4}x^{p^2}/2$. Thus we can use the subduction of the \tat\  $(\widehat{h_2}^{p^2},N_M(y)^{p^2+1})$
to construct an invariant $h_4$ with lead term $z^{p^4}$.
Define $\mathcal B':=\{x,N_M(y),\widehat{h_2},h_4\}$ and 
let $A$ denote the algebra generated by
$\mathcal B'$.
The only non-trivial \tat\ subducts to zero. Therefore 
$\mathcal B'$ is a SAGBI basis for $A$.
From Lemma~\ref{lem:field-sec9}, $\field[V_M]^E[x^{-1}]=\field[x,N_M(y),\bar{f}_{12346}][x^{-1}]$.
However, it follows from the definition of $\widehat{h_2}$
that $\field[x,N_M(y),\bar{f}_{12346}][x^{-1}]=\field[x,N_M(y),\widehat{h_2}][x^{-1}]$.
Thus, using Theorem~\ref{thm:compute}, 
$A=\field[V_{M}]^E$. Clearly $\LT(N_M(z))=z^{p^k}$ for $k\leq 4$. Since $\mathcal{B'}$ is a SAGBI basis for
$\field[V_E]^E$, this forces $k=4$, giving
$\LM(\mathcal{B})=\LM(\mathcal B')$.
\end{proof}

\section{The $\gamma_{1234}=0$, $\gamma_{1235}=0$, $\gamma_{1357}=0$ Strata}
\label{sec:1357=0,1234=0,1235=0}

In this section we consider representations $V_M$ 
for which $\gamma_{1235}(M)=0$, $\gamma_{1234}(M)=0$ and  $\gamma_{1357}(M)=0$.
For convenience, we write $\bg_{ijk\ell}$ for $\gamma_{ijk\ell}(M)$.
We assume that the first row of $M$ is non-zero; otherwise, the representation is of type $(2,1)$
and the calculation of $\field[V_M]^E$ can be found in \cite[\S 4]{CSW}.
Using a change of coordinates, see Proposition~4.3 of \cite{CSW}, the  ${\rm GL}_4(\field_p)$-action, and the hypothesis that 
$\bg_{1357}=0$,
we may take 
$$M=\begin{pmatrix} 
1&c_{12}&c_{13}&0\\
0&c_{22}&c_{23}&c_{24}
\end{pmatrix}.
$$

Since $\bg_{1235}=0$, either $c_{24}=0$ or $\{1,c_{12},c_{13}\}$ is linearly dependent over $\field_p$.
We assume $c_{24}\not=0$, otherwise the representation is not faithful and we can view $V_M$ as a 
representation of a group of rank $3$. Using the ${\rm GL}_4(\field_p)$ action, we replace the third
column by a linear combination of the first two columns to get
$$\begin{pmatrix} 
1&c_{12}&0&0\\
0&c_{22}&c_{23}&c_{24}
\end{pmatrix}.
$$
Expanding gives $$\bg_{1234}=(c_{12}-c_{12}^p)
\det
\begin{pmatrix} 
c_{23}&c_{24}\\
c_{23}^p&c_{24}^p
\end{pmatrix}.$$
Since $\bg_{1234}=0$, either $c_{12}\in\field_p$ or $\{c_{23},c_{24}\}$ is linearly dependent over $\field_p$.
However, if $\{c_{23},c_{24}\}$ is linearly dependent over $\field_p$, then the representation is not faithful. 
So we may assume $c_{12}\in\field_p$.  Using the ${\rm GL}_4(\field_p)$ action to replace the second column with
a linear combination of the first two columns gives
$$\begin{pmatrix} 
1&0&0&0\\
0&c_{22}&c_{23}&c_{24}
\end{pmatrix}.
$$
If  $\bg_{1246}=0$, then $\{c_{22},c_{23},c_{24}\}$ is linearly dependent over $\field_p$, and again the 
representation is not faithful.
Thus we may assume that  $\bg_{1246}\not=0$.
Using the above form for $M$, it is clear that $\bg_{1236}=0$, $\bg_{1346}=0$ and $\bg_{1246}=-\bg_{2346}$.
Thus
$$\bar{f}_{12346}=\bg_{1246}\left(y^px^{2p^2-p}-yx^{2p^2-1}\right)\in\field[x,y]^E.$$
Since $\field[x,y]^E=\field[x,N_M(y)]$, we have $N_M(y)=\bar{f}_{12346}/(\bg_{1246}x^{2p^2-p})=y^p-yx^{p-1}$.

\begin{lemma}\label{lem:field-sec10}
$\field[V_M]^E[x^{-1}]=\field[x,N_M(y),\bar{f}_{12468}][x^{-1}]$.
\end{lemma}
\begin{proof}
The proof is similar to the proof of Theorem~4.4  of \cite{CSW} 
(and Lemmas~\ref{lem:field-sec8} and \ref{lem:field-sec9}). 
Since $N_M(y)$ and $\bar{f}_{12468}/x^{p^3}$ are algebraically independent elements of
$\field[x,y,\delta/x]^E$
with $\deg(N_M(y))\deg(\bar{f}_{12468}/x^{p^3})=p^4=|E|$, applying Theorem~3.7.5 of \cite{DK} gives
$\field[x,y,\delta/x]^E=\field[x,N_M(y),\bar{f}_{12468}/x^{p^3}]$. 
The result then follows from the observation that
$\field[x,y,z]^E[x^{-1}]=\field[x,y,\delta/x]^E[x^{-1}]$.
\end{proof}

Subducting $\bar{f}_{12468}$ gives
$$\widetilde{h_1}:=\bar{f}_{12468}+\bg_{1246}\left(N_M(y)^{2p^2}+2N_M(y)^{p^2+p}x^{p^3-p^2}
                  +2N_M(y)^{p^2+1}x^{p^3-p}\right).$$

\begin{lemma} $\LT(\widetilde{h_1})=-2\bg_{1246}x^{p^3-1}y^{p^3+1}$.
\end{lemma}
\begin{proof} We work modulo the ideal $\langle x^{p^3}\rangle$. Using the definition,
$\bar{f}_{12468}\equiv -\bg_{1246}y^{2p^3}$. Since $N_M(y)=y^p-yx^{p-1}$, we have
$$N_M(y)^{2p^2}= y^{2p^3}-2y^{p^3+p^2}x^{p^3-p^2}+y^{2p^2}x^{2p^3-2p^2}\equiv  y^{2p^3}-2y^{p^3+p^2}x^{p^3-p^2}.$$
Expanding and simplifying gives
$$N_M(y)^{p^2+p}x^{p^3-p^2}+N_M(y)^{p^2+1}x^{p^3-p}\equiv y^{p^3+p^2}x^{p^3-p^2}-y^{p^3+1}x^{p^3-1}.$$
Thus 
\begin{eqnarray*}
\widetilde{h_1} &=& \bar{f}_{12468}+\bg_{1246}\left(N_M(y)^{2p^2}+2N_M(y)^{p^2+p}x^{p^3-p^2}
                  +2N_M(y)^{p^2+1}x^{p^3-p}\right)\\
                &\equiv& -2\bg_{1246}x^{p^3-1}y^{p^3+1}
\end{eqnarray*}
\end{proof}

Define $h_1:=-\widetilde{h_1}/(2\bg_{1246}x^{p^3-1)}$ so that $\LT(h_1)=y^{p^3+1}$. Note
that $$\field[x,N_M(y),h_1][x^{-1}]=\field[x,N_M(y),\bar{f}_{12468}][x^{-1}].$$

\begin{lemma}\label{lem:sub10}
Subducting the \tat\ $(h_1^p,N_M(y)^{p^3+1})$ gives an invariant with lead monomial
$x^pz^{p^4}$.
\end{lemma}
\begin{proof}
Refining the calculation in the proof of the previous lemma gives
$$\widetilde{h_1}\equiv_{\langle x^{p^3+1},\, x^{p^3}y\rangle}\bg_{1246}\left(-2y^{p^3+1}x^{p^3-1}+x^{p^3}z^{p^3}\right).$$
Thus 
$$h_1\equiv_{\langle x^2,\, xy\rangle}y^{p^3+1}-\frac{z^{p^3}x}{2}\hspace{10mm}{\rm and}\hspace{10mm}
h_1^p\equiv_{\langle x^{p+1},\, x^py\rangle}y^{p^4+p}-\frac{z^{p^4}x^p}{2}.$$
Furthermore
$$N_M(y)^{p^3+1}\equiv_{\langle x^{p+1},\, x^py\rangle}y^{p^4+p}-y^{p^4+1}x^{p-1}$$
and
$$h_1N_M(y)^{p^3-p^2}x^{p-1}\equiv_{\langle x^{p+1},\, x^py\rangle}y^{p^4+1}x^{p-1}.$$
Thus $\LT(h_1^p-N_M^{p^3+1}-h_1N_M(y)^{p^3-p^2})=-x^pz^{p^4}/2$.
\end{proof}
\begin{theorem}
If $\gamma_{1234}(M)=0$, $\gamma_{1235}(M)=0$, $\gamma_{1357}(M)=0$, and $\gamma_{1246}(M)\not=0$, 
then the set $\mathcal B:=\{x,N_M(y),h_1, N_M(z)\}$ is a SAGBI basis
for $\field[V_{M}]^E$. Furthermore, $\field[V_{M}]^E$ is a hypersurface 
with the relation coming from the subduction of the \tat\
 $(h_1^p,N_M(y)^{p^3+1})$.
\end{theorem}

\begin{proof}
Use the subduction given in Lemma~\ref{lem:sub10} to construct 
an invariant $h_2$ with lead term $z^{p^4}$. 
Define $\mathcal B':=\{x,N_M(y),h_1,h_2\}$ and 
let $A$ denote the algebra generated by
$\mathcal B'$.
The single non-trivial \tat\ for $\mathcal B'$
subducts to $0$ using Lemma~\ref{lem:sub10}.
Thus  $\mathcal B'$ is a SAGBI basis for $A$.
From Lemma~\ref{lem:field-sec10}, $\field[V_M]^E[x^{-1}]=\field[x,N_M(y),h_1][x^{-1}]$.
Thus, using Theorem~\ref{thm:compute},
$A=\field[V_{M}]^E$. Clearly $\LT(N_M(z))=z^{p^k}$ for $k\leq 4$. Since $\mathcal{B'}$ is a SAGBI basis for
$\field[V_E]^E$, this forces $k=4$, giving
$\LM(\mathcal{B})=\LM(\mathcal B')$.
\end{proof}

%%%%%%%%%%%%%%%%%%%%%%%%%%%%%%%%%%%%%%%%%%%%%%%%%%%%%%%%%%%%%%%%%%%%%%%%%%%%%%%%%%%%%%%%%%%%%%%%%%%%%%%%%

\end{document}